\documentclass[11pt]{article}
\usepackage{amsmath,amssymb,amsthm,bm,mathrsfs}
\numberwithin{equation}{section}

\usepackage{relsize,exscale}
\usepackage{graphicx}
\graphicspath{{./figures/}}
\usepackage[lofdepth,lotdepth]{subfig}
\usepackage{float}
\usepackage{enumitem}

\usepackage{xcolor}
\usepackage{hyperref}
\hypersetup{ 
	pagebackref,
    plainpages=false,
    colorlinks,
    linktocpage=true,   
    linkcolor={red!50!black},
    citecolor={blue!50!black},
    urlcolor={blue!80!black}
} 
\usepackage[hyperpageref]{backref}

\usepackage[margin=1.5in]{geometry}

\title{A Recovery-Based A Posteriori Error Estimator \\
for $\bm{H}(\mathbf{curl})$ Interface Problems
\thanks{This work was supported in 
part by the National Science Foundation
under grants DMS-1217081, DMS-1320608, and DMS-1522707.}}

\date{}

\author{Zhiqiang Cai
\thanks{Department of Mathematics, Purdue University, 150 N. University
Street, West Lafayette, IN 47907-2067, zcai@math.purdue.edu.}
\and Shuhao Cao
\thanks{Department of Mathematics, Pennsylvania State 
University, University Park, State College, PA 16802, scao@psu.edu.}}

\begin{document}

\newcommand{\vect}[1]{\bm{#1}}
\renewcommand{\a}{\alpha}
\newcommand{\G}{\Gamma}
\newcommand{\g}{\gamma}
\renewcommand{\k}{\kappa}
\newcommand{\D}{\Delta}
\renewcommand{\b}{\beta}
\renewcommand{\d}{\delta}
\newcommand{\lam}{\lambda}
\newcommand{\e}{\epsilon}
\newcommand{\vare}{\varepsilon}
\newcommand{\z}{\zeta}
\newcommand{\s}{\sigma}
\newcommand{\Om}{\Omega}
\renewcommand{\S}{\Sigma}
\newcommand{\om}{\omega}
\newcommand{\RR}{\mathbb{R}}
\newcommand{\sP}{\mathscr{P}}
\newcommand{\vC}{\vect{C}}
\newcommand{\vH}{\vect{H}}
\newcommand{\vL}{\vect{L}}
\newcommand{\vP}{\vect{P}}
\newcommand{\vX}{\vect{X}}
\newcommand{\va}{\vect{a}}
\newcommand{\vb}{\vect{b}}
\newcommand{\vc}{\vect{c}}
\newcommand{\ve}{\vect{e}}
\newcommand{\vu}{\vect{u}}
\newcommand{\vv}{\vect{v}}
\newcommand{\vz}{\vect{z}}
\newcommand{\vf}{\vect{f}}
\newcommand{\vg}{\vect{g}}
\newcommand{\vn}{\vect{n}}
\newcommand{\vp}{\vect{p}}
\newcommand{\vt}{\vect{t}}
\newcommand{\vs}{\vect{s}}
\newcommand{\vx}{\vect{x}}
\newcommand{\vw}{\vect{w}}
\newcommand{\bnu}{\vect{\nu}}
\newcommand{\bsig}{\vect{\sigma}}
\newcommand{\btau}{\vect{\tau}}
\newcommand{\bphi}{\vect{\phi}}
\newcommand{\bvphi}{\vect{\varphi}}
\newcommand{\bpsi}{\vect{\psi}}

\newcommand{\cA}{\mathcal{A}}
\newcommand{\cB}{\mathcal{B}}
\newcommand{\cD}{\mathcal{D}}
\newcommand{\cI}{\mathcal{I}}
\newcommand{\cT}{\mathcal{T}}
\newcommand{\cN}{\mathcal{N}}
\newcommand{\cE}{\mathcal{E}}
\newcommand{\cF}{\mathcal{F}}
\newcommand{\cM}{\mathcal{M}}
\newcommand{\cO}{\mathcal{O}}
\newcommand{\cP}{\mathcal{P}}
\newcommand{\cR}{\mathcal{R}}
\newcommand{\cS}{\mathcal{S}}
\newcommand{\fI}{\mathfrak{I}}
\newcommand{\bphit}{\bm{\phi}_{\top}}

\newcommand{\p}{\partial}
\newcommand{\wt}{\widetilde}
\newcommand{\wh}{\widehat}
\newcommand{\ol}{\overline}
\newcommand{\mcup}{\mathsmaller{\,\bigcup}}
\newcommand{\mcap}{\mathsmaller{\,\bigcap}}
\newcommand{\nab}{\nabla}
\newcommand{\curl}{\mathbf{curl}\hspace{0.7pt}}
\renewcommand{\div}{\mathrm{div}}
\newcommand{\conv}{\mathrm{conv}}
\newcommand{\divv}{\nab\!\cdot\!}
\newcommand{\cross}{\!\times\!}
\newcommand{\curlt}{\nab\cross}

\newcommand{\sprod}{\mathop{\mathchoice  
{\textstyle\prod}  {\prod} {\prod} {\prod} }\nolimits}
\newcommand{\vsprod}{\mathop{\mathchoice  
{\textstyle\bm{\prod}}  {\bm{\prod}} {\bm{\prod}} {\bm{\prod}} }\nolimits}
\newcommand{\osc}{\mathrm{osc}}

\newcommand{\norm}[1]{\left\Vert#1\right\Vert}
\newcommand{\abs}[1]{\left| {#1}\right|}

\makeatletter
\newcommand{\enorm}{\@ifstar\@enorms\@enorm}
\newcommand{\@enorms}[1]{%
\left\vert\mkern-2mu\left\vert\mkern-2mu\left\vert
#1
\right\vert\mkern-2mu\right\vert\mkern-2mu\right\vert
}
\newcommand{\@enorm}[2][]{%
\mathopen{#1\vert\mkern-2mu#1\vert\mkern-2mu#1\vert}
#2
\mathclose{#1\vert\mkern-2mu#1\vert\mkern-2mu#1\vert}
}
\makeatother

\newcommand{\avg}[2]{\{#1\}_{\raisebox{-2.4pt}{\scriptsize$#2$}}}
\newcommand{\jump}[2]{%
{\lbrack\kern-1.4pt\lbrack
#1
\rbrack\kern-1.4pt\rbrack}_{\raisebox{-1pt}{\scriptsize$#2$}}
}

\newcommand{\oversett}[2]{%
\mathop{#2}\limits^{\vbox to -.1ex{\kern -0.45ex\hbox{$\scriptstyle #1$}\vss}}}
\newcommand{\at}[1]{\big\vert_{\raisebox{-0.5pt}{\scriptsize$#1$}} }
\newcommand{\binprod}[2]{\bigl( {#1},{#2} \bigr)}

\newcommand{\Lt}{L^2(\Om)}
\newcommand{\vLt}{\vL^2(\Om)}
\newcommand{\Ho}{H^1(\Om)}
\newcommand{\Hoz}{H^1_0(\Om)}
\newcommand{\vHcrl}{\vH(\curl;\Om)}
\newcommand{\vhcrl}{\vH(\curl)}
\newcommand{\vHcrlz}{\vH_0(\curl;\Om)}
\newcommand{\vHcrlg}{\vH_{\vg}(\curl;\Om)}
\newcommand{\vHcrlzz}{{\oversett{\circ}{{}\vH}}_0(\curl;\Om)}
\newcommand{\vhdiv}{\vH(\div)}
\newcommand{\vHdiv}{\vH(\div;\Om)}
\newcommand{\vHdivw}[1]{\vH(\div\, #1;\Om)}
\newcommand{\vHs}[2]{\vH^{#1}(#2)}
\newcommand{\PHs}[1]{{P\!H}^{#1}(\Om,\mathscr{P})}
\newcommand{\vPH}[1]{{P\!\vH}^{#1}(\Om,\mathscr{P})}
\newcommand{\vPHo}{{P\!\vH}^1(\Om,\mathscr{P})}
\newcommand{\vPCinf}{{P\hspace{-1pt}\vC}^{\infty}(\Om,\sP)}
\newcommand{\ND}{\vect{\cN\!\cD}}
\newcommand{\PC}{\vect{\cP}_1}
\newcommand{\RT}{\vect{\cR\!\cT}}
\newcommand{\BDM}{\vect{\cB\!\cD\!\cM}}

\newtheorem{theorem}{Theorem}[section]
\numberwithin{theorem}{section}
\newtheorem{assumption}[theorem]{Assumption}

\newtheorem{remark}[theorem]{Remark}
\theoremstyle{definition}
\newtheorem{definition}[theorem]{Definition}
\newtheorem{lemma}[theorem]{Lemma}


\maketitle

\begin{abstract}
This paper introduces a new recovery-based a posteriori error 
estimator for the lowest order N\'{e}d\'{e}lec finite element approximation to 
the $\vhcrl$ interface problem. The error estimator is analyzed by establishing 
both the reliability and the efficiency bounds and is supported by numerical
results. Under certain assumptions, it is proved 
that the reliability and efficiency constants are independent of the jumps of the 
coefficients.
\end{abstract}

%

\thispagestyle{plain}
\markboth{ZHIQIANG CAI AND SHUHAO CAO}
{RECOVERY ERROR ESTIMATORS FOR H(CURL) PROBLEM}

\section{Introduction}
Let $\Om\subset\RR^3$ be a bounded polyhedral domain with 
Lipschitz boundary. 
Let $\mathscr{P}=\{\Om_j\}^m_{j=1}$ be a partition of the domain $\Om$
with each subdomain $\Om_j$ being polyhedron. The collection of 
interfaces $\big(\bigcup^m_{j=1}\p\Om_j\big)\backslash \p \Om$ is denoted by 
$\fI$. Assume that $\mu$ and $\b$ are piecewise, positive constants with 
respect to the partition $\mathscr{P}$. We consider the
following $\vhcrl$ interface problem:
\begin{equation}
\label{eq:pb-model}
\left\{
\begin{aligned}
\curlt (\mu^{-1} \curlt \vu) + \b \vu &= \vf, \quad \text{ in } \Om,
\\[2mm]
\vu \cross \vn &= \vg, \quad \text{ on } \p \Om,
\end{aligned} 
\right.
\end{equation}
where $\vn$ is the unit outward vector normal to the boundary of $\Om$. 
This model problem originates from a stable 
marching scheme of the second-order hyperbolic partial differential equation
on the electric field intensity $\vu$ that is resulted from the Maxwell 
equations (e.g. see \cite{Dautray-Lions,Hiptmair99}). The $\mu$ is the 
magnetic permeability, 
and the $\b \sim \frac{\e}{\D t^2}+ \frac{\s}{\D t}$ is related to the 
dielectric constant $\e$ and conductivity $\s$ scaled by the 
time-marching step size $\D t$. The boundary data $\vg$ is ``admissible'' in a 
sense that we will elaborate when introducing the finite element approximation 
\eqref{eq:pb-fem}.
Throughout this article, boldface letters stand for vector fields and spaces of 
vector fields, non-boldface letters stand for scalar functions and spaces of 
scalar functions.

The variational formulation of problem~\eqref{eq:pb-model} involves the 
Hilbert space $\vhcrl$, which is the collection of all square integrable 
vector fields whose curl are also square integrable:
\begin{equation}
\label{eq:space-hcurl}
 \vHcrl := \{\vv \in \vLt\,: \curlt \vv \in \vLt\}.
\end{equation}
The right hand side data $\vf$ 
depends on the original source current and on the electric field intensity 
at previous time steps in the time-marching scheme. In almost all relevant 
literatures, $\vf$ is assumed to be divergence free. In this paper, we assume 
that $\vf\in \vhdiv$, where $\vhdiv$ is the analog of the
$\vhcrl$ space for the divergence operator:
\begin{equation}
\label{eq:sp-hdiv}
 \vHdiv := \{\vv \in \vLt\,: \divv \vv \in \Lt\}.
\end{equation}

For the finite element approximation to~\eqref{eq:pb-model}, 
N\'{e}d\'{e}lec introduced the $\vhcrl$-conforming edge elements in 
\cite{Nedelec80}, which preserves the continuity of the tangential components, 
and certain \emph{a priori} error estimates can be established (e.g. see 
\cite{Monk}). However, the electromagnetic fields have limited regularities at
reentrant corners and material interfaces (see 
\cite{Costabel-Dauge,Costabel-Dauge-Nicaise}). Hence the assumptions for 
\emph{a priori} error estimates fail, and this is
where adaptive mesh refinement is introduced to perform local mesh refining
process within the regions that have relatively large approximation errors. 

The \emph{a posteriori} error estimation for the $\vhcrl$ problem 
in~\eqref{eq:pb-model} with constant or continuous coefficients has been 
studied recently by several researchers. 
Several types of \emph{a posteriori} error estimators have been introduced and 
analyzed. These include residual-based estimators and the corresponding 
convergence analysis (explicit \cite{Beck-Hiptmair-Hoppe-Wohlmuth,
Nicaise03, Nicaise07,Schoberl07,Chen-Xu-Zou-1,Chen-Xu-Zou-2, Chen10}, and 
implicit \cite{Harutyunyan08}), equilibrated estimators \cite{Braess06}, 
and recovery-based estimators \cite{Nicaise05}. It is interesting to note that 
there are four types of errors in the explicit residual-based estimator (see 
\cite{Beck-Hiptmair-Hoppe-Wohlmuth}). Two of them are 
standard, i.e., the element residual and the face jump associated with the 
original equation in \eqref{eq:pb-model}. The other two are also the 
element residual and the face jump, but associated with the divergence of the 
original equation: $\divv (\b \vu) =\divv \vf$.

The recovery-based estimator studied in \cite{Nicaise05} may be viewed as an
extension of the popular Zienkiewicz-Zhu (ZZ) error estimator (\cite{ZZ87}) for the 
Poisson equation to the $\vhcrl$ problem with constant coefficients. More 
specifically, two quantities related to the solution $\vu$ and $\curlt \vu$ are 
recovered based on the current approximation of the solution in a richer recovery 
space. The recovery space in \cite{Nicaise05} is the 
\emph{continuous} piecewise polynomial space, and the recovery procedure is done 
through averaging on vertex patches. 
The resulting ZZ estimator consisting of two terms is shown to be equivalent to the 
face jumps across the element faces. The element residuals are not included in the 
estimator in \cite{Nicaise05}.

The purpose of this paper is to develop and analyze an efficient, reliable, and 
robust recovery-based \emph{a posteriori} error estimator for the finite element 
approximation to the $\vhcrl$ interface problem, i.e., 
problem~\eqref{eq:pb-model} with piecewise constant coefficients. 
Theoretically, the efficiency refers that the local 
error indicator is bounded above by the local error, the reliability refers 
that the global error is bounded above by the global estimator. The robustness 
refers that constants in the efficiency and the reliability bounds are independent 
of the jumps of the coefficients.

The recovered-based estimator introduced in this paper may be 
viewed as an extension of our previous work in \cite{Cai09, Cai10} on the 
diffusion interface problem to the $\vhcrl$ interface problem, which partially 
resolve the non-robustness of ZZ error estimator for interface problems. 
Specifically, we recover two quantities related to $\mu^{-1}\curlt \vu$ and $\b \vu$ 
in the respective $\vhcrl$- and $\vhdiv$-conforming finite element spaces (the 
lowest order N\'{e}d\'{e}lec and Brezzi-Douglas-Marini elements respectively). For 
discussions on which quantities 
to be recovered and in what finite element spaces, see \cite{Cai09, Cai10}.
The resulting estimator measures the face jumps of the tangential 
components and the normal component of the numerical approximations to 
$\mu^{-1}\curlt \vu$ and $\b \vu$, respectively. Our study indicates that the 
element residual is no longer higher order than the rest terms in the 
estimator, the contrary of which is proved to be the case in the diffusion problem 
(\cite{Carstensen-Verfurth}). As a result, the 
element residual using recovered quantities is part of our proposed error estimator 
as well. 

Theoretically proving a robust reliability bound for the $\vhcrl$
interface problem is much harder than that for the diffusion interface problem. 
This is because one needs to estimate the dual norm of the residual functional
over the $\vhcrl$ space. To overcome this difficulty, one needs to use a 
Helmholtz decomposition of the error (e.g. see \cite{Beck-Hiptmair-Hoppe-Wohlmuth}). 
For the $\vhcrl$ interface problem, additional difficulty is that the 
decomposition has to be stable under a coefficient weighted norm. To obtain 
such a decomposition is a non-trivial matter, for the discrete level, a discrete 
weighted Helmholtz decomposition is studied together with its application and 
analysis for non-overlapping domain decomposition method in \cite{Hu-Shu-Zou}. 
Here for the continuous version, under certain assumptions, we are able to 
accomplish this task through establishing 
a weighted identity relating gradient, curl, and divergence of a piecewisely smooth 
vector field (see Lemma~\ref{lem:normeqiv}), which is an extension to the technique 
used in \cite{Costabel-Dauge-eigenvalue}. The final decomposition result in our 
paper is similar to the one in \cite{Zhu11}, the coefficient distribution setting is 
slightly more general than the one used in \cite{Zhu11}. Our proof uses a piecewise 
regularity result from \cite{Costabel-Dauge-Nicaise} than directly building from 
extension in \cite{Zhu11}. 

Another necessary tool for proving a robust 
reliability bound is a tweaked version of the Cl\'{e}ment-type interpolation. We
are able to extend naturally from the idea in \cite{Bernardi-Verfurth, Petzoldt02} 
for the vertex-based continuous Lagrange element to the edge-based 
N\'{e}d\'{e}lec element in three dimensions. Moreover, our quasi-monotone assumption 
on the distribution of the coefficients is based on edges which is similar to that 
of \cite{Petzoldt02} based on vertices in 2D, 
and our proof borrows the idea from \cite{Bernardi-Verfurth}. This is the first 
N\'{e}d\'{e}lec interpolation known to achieve such a robust bound.

Moving onto the efficiency bound estimate, we prove the every part of the local 
recovery-based error estimator can be bounded by the robustly weighted 
residual-based error estimator.
For the local error indicator measuring the irrotational part (gradient part) of the 
error, using the weighted averaging technique, the constant in the bounds  is 
independent of the coefficient jumps across the interface. 
For the local estimator measuring the jump of the numerical approximation to 
$\mu^{-1}\curlt \vu$ (weak solenoidal or curl part of the error), 
the degrees of freedom of the corresponding recovered quantity sits on the edge 
(lowest order N\'{e}d\'{e}lec elements). 
Consequently, the averaging is performed within an edge patch in 3D, and this 
resembles the averaging of Zienkiewicz-Zhu (ZZ) error estimator on vertex patch in 
2D. It is known that, if the averaging is done in vertex patches that span across 
the interface, ZZ error estimator (even correctly 
weighted) is not robust with respect to the ratio of the max/min of the coefficients 
on diffusion interface problem (e.g. see \cite{Cai09}). Here for the $\vhcrl$ 
interface problem, under the assumption of the quasi-monotone distribution of the 
coefficients again, we are able to prove that the ZZ type averaging, if carefully 
weighted, yields a robust efficiency bound with respect to the coefficient jump. 
This is a first known result as well.

Numerically, we are able to show that the recovery-based estimator studied in 
this paper is more accurate than the residual-based estimator in 
\cite{Beck-Hiptmair-Hoppe-Wohlmuth} for several test problems.

This paper is organized as follows. The
variational formulation and
the $\vhcrl$-conforming finite element approximation are introduced in
section 2. The explicit recovery procedures and the resulting
\emph{a posteriori} error estimators are discussed in section 3. In
section 4, the reliability and efficiency bounds are proved along with
the technical tools for analysis. The proofs of the bounds for the weighted 
Helmholtz decomposition is presented in the appendix if certain conditions are 
met. Finally, some numerical results for the benchmark testing problems are 
presented in section 5.

\section{Preliminaries}
\subsection{Notations}

Hereby we list some formal definitions concerning 
problem~\eqref{eq:pb-model}. The function space for the variational problem is:
\begin{equation}
\label{eq:sp-hcurlg}
\vHcrlg := \{\vu \in \vHcrl\,: \vu \cross \vn = \vg
\text{ on } \p \Om\},
\end{equation}
equipped with the $\vhcrl$ norm
\begin{equation}
\label{eq:norm-hcurl}
\norm{\vu}_{\vhcrl} = \Bigl(\norm{\vu}^2_{\vLt} + \norm{\curlt
\vu}^2_{\vLt} \Bigl)^{1/2} .
\end{equation}

The bilinear form of the variational problem 
to~\eqref{eq:pb-model} is:
\begin{equation}
 \label{eq:bilinearform}
\cA(\vu,\vv) := \int_{\Om} (\mu^{-1} \curlt \vu \cdot \curlt \vv +
\b\, \vu\cdot \vv)d\vx,
\end{equation}
and the coefficient-weighted norm related to this problem is:
\begin{equation}
 \label{eq:norm-energy}
\enorm{\vu}^2 := \cA(\vu,\vu).
\end{equation}
If a subscript is added for the weighted norm, it means the local weighted 
norm defined on an open subset $\cO\subset \Om$:
\begin{equation}
\label{eq:norm-local}
\enorm{\vu}_{\cO}^2 := \int_{\cO} (\mu^{-1} \curlt \vu \cdot \curlt \vu +
\b\, \vu\cdot \vu)d\vx.
\end{equation}

In addition to the standard $\vhdiv$ space \eqref{eq:sp-hdiv}, we need the 
weighted version as well:
\begin{equation}
\label{eq:sp-hdivw}
\vHdivw{\a} = \{\vv\in \vLt: \divv (\a \vv) 
\in \Lt \text{ in } \Om \}.
\end{equation}

Let $\cT_h = \cup \{K\}$ be a triangulation of $\Om$ using tetrahedra 
elements. The sets of all the vertices, edges, and faces of this triangulation 
are denoted by $\cN_h$, $\cE_h$, and $\cF_h$, respectively. Denote the vertices, 
edges, and faces being subsets or elements of a geometric object $\cM$ by 
$\cN_h(\cM)$, $\cE_h(\cM)$, and $\cF_h(\cM)$, where $\cM$ can be an element from the 
objects in the simplicial complex of the triangulation like a specified element $K$, 
or the whole boundary $\p\Om$, etc.
For any vertex $\vz \in \cN_h$, let $\lam_{\vz}$ be the nodal basis function of 
continuous piecewise linear element associated with the vertex $\vz$.
 
A fixed unit normal vector $\vn_F$ is assigned to each face $F$, and a fixed 
unit tangential vector $\vt_e$ to each edge $e$. For any 
scalar- or vector-valued function $v$, define 
$\jump{v}{F} = v^- - v^+$ on an interior face $F\in \cF_h$ with a 
fixed unit normal vector $\vn_F$, where 
$v^{\pm} = \lim\limits_{\e\to 0^{\pm}}v(\vx+\e \vn_F)$.
Define $\avg{v}{F} = (v^+ + v^-)/2$ as the average on 
this face $F$. If $F$ is a boundary face, the function $v$ is extended by zero 
outside the domain to compute $\jump{v}{F}$ and $\avg{v}{F}$.

The following algebraic identity is handy later in 
proving identities involving interfaces for any scalar- or vector-valued 
quantities $a$ and $b$:
\begin{equation}
\label{eq:dcp-jump}
\jump{ab}{F} = \avg{a}{F}\jump{b}{F} + \jump{a}{F} \avg{b}{F}.
\end{equation}

Denote the diameter of an element
$K\in \cT_h$ by $h_K$ and the diameter of a face $F\in \cF_h$ by $h_F$. We 
assume that the triangulation $\cT_h$ is shape regular (see \cite{Ciarlet}), 
and this assumption holds for any tetrahedron during the local mesh refining 
process.

The following notations serve as the languages to describe the local element or face 
patches. They will be used later in local weighted 
recovery procedure (Section 3), and in the proof of estimates for the 
weighted Cl\'{e}ment-type interpolation (Section 4).

\begin{figure}[htb]
\begin{center}
\subfloat[For edge $e$ with the unit tangential vector $\vt_e$, 
$\om_e = \cup_{i=1}^4 K_i$, $\om_{e,F} = \cup_{i=1}^4 F_i$.]
{\includegraphics[scale=1]{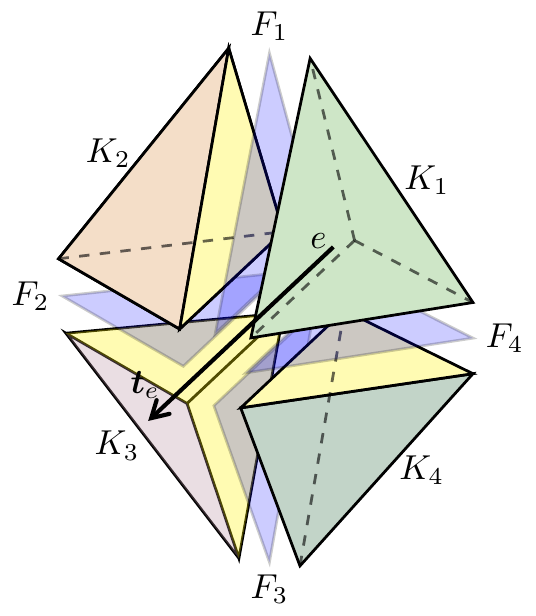}
\label{fig:geom-patche}
}
\hspace{0.5in}
\raisebox{0.5cm}{
\subfloat[$\om_F = K_+\cup K_-$ with the unit normal vector $\vn_F$.]
{\includegraphics[scale=1]{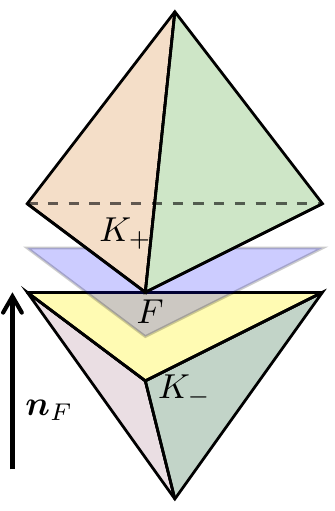}
\label{fig:geom-patchf}
}
}
\end{center}

\caption{A dissection view of the local edge patch $\om_e$, $\om_{e,F}$, 
and face patch $\om_F$}
\label{fig:geom-patch}
\end{figure}

For a face $F\in \cF_h$, let $\om_F$ be the patch of 
the tetrahedra sharing this face $F$. Let $\om_{K,F}$ be the patch of the 
tetrahedra sharing a face with $K$.

For an edge $e\in \cE_h$, denote by 
\[
\om_e = \mcup_{\{K\in \cT_h: \,e\in \cE_h(K) \}} K
\]
the collection of all elements having $e$ as a common edge, where 
$\cE_h(K)$ is the collection of edges of the element $K$.
For the edge patch $\omega_e$, we define two $\mu$-weighted edge patches associated 
with an edge $e\in\cE_h$, which can be understood as the collection of the elements 
with the biggest/smallest $\mu^{-1}$ on an edge patch, are referred to 
\begin{equation}
\label{eq:patch-e}
\begin{gathered}
\wt{\om}_e =  \mcup_{K\in \cI_e} K, \; \text{ where }\;  
\cI_e = \{K\subset \om_e: \mu_{K}^{-1}=\max_{K'\subset \om_e} \mu_{K'}^{-1} \},
\\
\text{ and }\quad \wh{\om}_e =  \mcup_{K\in \cI_e} K, \; \text{ where }\;  
\cI_e = \{K\subset \om_e: \mu_{K}^{-1}=\min_{K'\subset \om_e} \mu_{K'}^{-1} \}.
\end{gathered}
\end{equation}
Denote the union of the interior faces within an edge patch as follows:
\begin{equation}
\label{eq:patch-ef}
\om_{e,F} = \mcup_{F\in \cF_h(\om_e)} F \backslash\, \p \om_e.
\end{equation}
Using Figure \ref{fig:geom-patch} as an illustration, $\om_{e,F} = \cup_{i=1}^4 F_i$.
Define the $\mu$-weighted patch of interior faces as follows:
\begin{equation}
\label{eq:patch-wtef}
\wh{\om}_{e,F} = \mcup_{F\in \cI_e} F, \; \text{ where }\;  
\cI_e = \{F\in \cF_h(\wh{\om}_e): F\subset \om_{e,F}\}.
\end{equation}
Taking Figure \ref{fig:geom-patch} as an example again, if $\wh{\om}_e = K_1\cup 
K_2$, then $\wh{\om}_{e,F} = F_1\cup F_2 \cup F_4.$ 

For an element $K\in\cT_h$, denote the patch of all elements sharing an edge 
with $K$ by
\begin{equation}
\label{eq:patch-Ke}
\om_{K,e} = \mcup_{K\in \cI_{K,e}} K, \; \text{ where }\; 
\cI_{K,e} = \{K\in \cT_h: K \subset \om_e\mbox{ with } e\in \cE_h(K) \}.
\end{equation}

Similarly, for a vertex $\vz\in \cN_h$, denote by
\[
\om_{\vz} = \mcup_{\{K\in \cT_h: \;\vz\in \cN_h(K) \}} K
\]
the collection of all elements having $\vz$ as a common vertex. For the vertex 
patch $\omega_{\vz}$, the $\b$-weighted edge patch associated with an 
vertex $\vz\in\cN_h$ is referred to 
\begin{equation}
\label{eq:patch-z}
\wt{\om}_{\vz} = \mcup_{K\in \cI_{\vz}} K, \; \text{ where }\; 
\cI_{\vz} = \{K\subset \om_{\vz}: \b_K=\max_{K\subset \om_{\vz}} \b_K \}.
\end{equation}
For an element $K\in\cT_h$, denote the patch of all elements sharing a vertex 
with $K$ by
\begin{equation}
\label{eq:patch-Kz}
\om_{K,\vz} = \mcup_{K\in \cI_{K,\vz}} K, \; \text{ where }\; 
\cI_{K,\vz} = \{K\in \cT_h: K \subset \om_{\vz}\mbox{ with } \vz\in \cN_h(K)\}.
\end{equation}

\subsection{Finite Element Approximation}
The corresponding variational formulation 
of~\eqref{eq:pb-model} is
\begin{equation}
\label{eq:pb-variational}
\left\{
\begin{aligned}
& \text{Find } \vu\in \vHcrlg \text{ such that:}
\\[2mm]
& \cA(\vu,\vv) = (\vf,\vv), \quad \forall\, \vv \in \vHcrlz,
\end{aligned}
\right.
\end{equation}
where $(\cdot,\cdot)$ is the standard $\vLt$-inner product. Because of $\mu$ 
and $\b$ being uniformly positive on the domain, the coefficient-weighted
norm~\eqref{eq:norm-energy} is equivalent to the graph norm \eqref{eq:norm-hcurl} 
for $\vHcrl$. Moreover the bilinear form \eqref{eq:bilinearform} is intrinsically 
coercive with respect to this norm. 
By the Lax-Milgram lemma, there exists a unique solution in $\vHcrlg$ to 
problem~\eqref{eq:pb-variational} when the boundary data is ``admissible''.

The solution  $\vu$ of~\eqref{eq:pb-variational} is approximated in a
$\vHcrl$-conforming finite element space: the lowest order N\'{e}d\'{e}lec 
finite element space $\ND_0$ (see \cite{Nedelec80}). On each element $K$, define
\[
\ND_0(K) = \{\vp(\vx) \in \vP_1(K): \vp = \va + \vb \cross \vx,
\, \va, \vb\in \RR^3\}.
\]
The global finite element space $\ND_0$ is glued together through the
continuity condition of $\vHcrl$:
\[
\ND_0 = \{\vp \in \vHcrl: \vp(\vx)\big|_K \in \ND_0(K) 
\quad \forall K\in \cT_h\}.
\]
For simplicity, we assume that the Dirichlet boundary data can be represented 
as the tangential trace of an $\ND_0$ vector field, i.e., $\vg = \vp\cross \vn$ 
on the boundary, where $\vp \in \ND_0$. The finite element approximation is
\begin{equation}
 \label{eq:pb-fem}
\left\{
\begin{aligned}
& \text{Find } \vu_h\in \ND_0\cap \vHcrlg \text{ such that:}
\\[2mm]
& \cA(\vu_h,\vv_h) = (\vf,\vv_h), \quad 
\forall\, \vv_h \in \ND_0\cap \vHcrlz.
\end{aligned}
\right.
\end{equation}
The problem in~\eqref{eq:pb-fem} is well-posed in its own right.

Before building the error estimator, we need an $\vHdiv$-conforming finite 
element space $\BDM_1$, which is the linear order Brezzi-Douglas-Marini face 
element (see \cite{Brezzi-Fortin}). On each element $K$, define, in a way that leads 
to the local basis construction,
\[
\BDM_1(K) = \{\vp(\vx) \in \vP_1(K): \vp = \va + c \,\vx 
+ \curlt(b\vs) , \; \va,\vs \in\RR^3, c\in \RR^1, b\in B(K) \},
\]
where $B(K)$ the space of quadratic edge bubble functions in $K$. Similarly 
the global $\BDM_1$ inherits the continuity condition from $\vHdiv$:
\[
\BDM_1 = \{\vp \in \vHdiv: \vp(\vx)\big|_K \in \BDM_1(K) 
\quad \forall K\in \cT_h\}.
\] 
Let $K\in\cT_h$ be an arbitrary tetrahedral element with vertices $\vz_i$,  
$\vz_j$, $\vz_k$, and $\vz_l$, and
let $\vn_{i}$ be the outer unit vector normal to the face
$F_i = \conv( \vz_j \vz_k \vz_l)$, opposite to the vertex $\vz_i$.
Let $\vt_{ij}$ be the unit 
vector of the edge $e_{ij}$ orienting in the direction of $\vz_j - \vz_i$.
The $\ND_0$ nodal basis function for the edge $e_{ij}$ can 
be written as (e.g. see~\cite{Solin-Segeth-Dolezel,Whitney}) :
\begin{equation}
\label{eq:basis-nd0}
\bvphi_{e_{ij}} =
\frac{\lam_i  \vn_j}{\vt_{ij}\cdot \vn_j} 
- \frac{\lam_j  \vn_i}{\vt_{ij}\cdot \vn_i} ,
\end{equation}
where $\lam_n$ for $n=i,\,j,\,k,\,l$ are the barycentric coordinates associated 
with the vertex $\vz_n$ satisfying $\lam_i + \lam_j + \lam_k + \lam_l=1$. The 
degree of freedom of $\ND_0$ can be then associated with each edge $e\in 
\cE_h$, in that $\bvphi_{e}$ satisfies
\[
\bvphi_{e} \cdot \vt_{e'}\at{e'}= 
\frac{1}{|e'|}\int_{e'} \bvphi_{e} \cdot \vt_{e'} \,ds=  \pm\d_{e e'},
\quad \forall e'\in \cE_h(K),
\]
where $\d_{e e'}$ is the Kronecker delta. The plus sign is taken when locally 
$\vt_{e}$'s direction coincides with the globally fixed $\vt_{e'}$'s. 

Similarly, we cook up a customized version 
of the $\BDM_1$ nodal basis function associated with the vertex $\vz_j$ on face 
$F_i$ as follows
\begin{equation}
\label{eq:basis-bdm}
\bpsi_{F_i,\vz_j} = 
9\frac{\lam_j\vt_{ij}}{\vn_{i}\cdot \vt_{ij}}
- 3\frac{\lam_k\vt_{ik}}{\vn_{i}\cdot \vt_{ik}}
- 3\frac{\lam_l\vt_{il}}{\vn_{i}\cdot \vt_{il}}.
\end{equation}
Now the degrees of freedom of $\BDM_1$ can be defined using the first moment 
on $F_i$ because $\bpsi_{F_i,\vz_j}$ satisfies:
\[
\frac{1}{|F_i|}\int_{F_i} (\bpsi_{F_i,\vz_j}\cdot \vn_{F_i})\lam_{\vz_m}\,dS = 
\pm\delta_{jm}, \quad \text{ for } m\in \{j,k,l\},
\]
and $\bpsi_{F_i,\vz_j}\cdot \vn_{F'} = 0$ for any face $F'\in \cF_h(K)$ other 
than $F_i$. Similarly, the plus sign is taken when locally the exterior unit normal 
$\vn_i$ to face $F_i$ is in the same direction with the globally fixed unit normal 
$\vn_{F_i}$.

\section{The Recovery-type Error Estimator}

There are two important physical quantities of interest: 
the magnetic field intensity and displacement current
density which are related to the electric field intensity 
$\vu$. The magnetic field intensity at current time step is denoted by $\bsig$, 
and displacement current density diluted by the time step size is denoted by 
$\btau$. For $\vhcrl$ problem~\eqref{eq:pb-model} that is time-independent, they can 
be represented by the following
\begin{equation}
\label{eq:constitutive}
\bsig = \mu^{-1} \curlt \vu \; \text{ and } \; \btau = \b \vu,
\end{equation}
then the partial differential equation in \eqref{eq:pb-model} 
can be rewritten as 
 \begin{equation}\label{eq:equilibrium}
 \curlt \bsig + \btau  = \vf. 
 \end{equation}
By the assumption of the 
data $\vf\in \vHdiv$, it is straightforward to verify that the $\vu$, 
$\bsig$, and $\btau$ lie in the following spaces
\begin{equation}
\label{eq:sp-cont}
\begin{gathered}
\vu \in \vHcrl, \quad
\btau \in  \vHdiv,
\\ 
\text{and }\;\bsig \in \vHcrl\cap \vHdivw{\mu},
\end{gathered}
\end{equation}
where the weighted space $\vHdivw{\mu}$ is defined in \eqref{eq:sp-hdivw}.

The relation \eqref{eq:sp-cont} indicates the continuity conditions $\vu$, 
$\bsig$, and $\btau$ must fulfill in the continuous level in \eqref{eq:constitutive} 
and \eqref{eq:equilibrium}. These continuity 
requirements not just come from the operator theory in Hilbert spaces, but 
also translate from the original Maxwell equations. For an 
arbitrary interface $S$ within the domain, if there 
is no surface charge on $S$, it is well known from physics 
(e.g., see~\cite{Monk}) that
\begin{equation}
\label{eq:jump}
\jump{\vu\cross\vn}{S} = \bm{0}, \;\; \jump{\bsig\cross\vn}{S} = \bm{0}, 
\;\;
\jump{\mu\bsig\cdot\vn}{S} = 0,
\;\;\text{and}
\;\;
\jump{\btau\cdot \vn}{S} = 0. 
\end{equation}
These zero-jump conditions are consistent with 
the continuity conditions for $\vhcrl$ and $\vhdiv$, respectively. 
However, during rendering the continuous problem into the finite 
element approximation, numerical magnetic field intensity 
and numerical displacement current density,
\[ 
 \bsig_h := \mu^{-1}\curlt \vu_h
 \quad\text{and}\quad
 \btau_h := \b \vu_h,
\] 
violate the second and the last continuity
conditions from \eqref{eq:jump}, respectively. 
Therefore, two quantities are recovered in the respective $\vhcrl$-conforming 
and $\vhdiv$-conforming finite element spaces using an explicit local weighted 
averaging technique.
Note that that the normal component of $\btau_h$ is 
a piecewise polynomial of degree one on each face. Consequently, we need to use 
either $\BDM_1$ or $\RT_1$ for recovering displacement current density, instead 
of $\RT_0$.

\subsection{Local Recovery Procedure}
We recover two quantities, $\bsig^*_h$ and $\btau^*_h$, based on $\bsig_h$ and 
$\btau_h$ through weighted averaging locally on edge and face patches, respectively. 
To this end, for a fixed interior face $F\subset \cF_h$, denote by $K_{\pm}$ the 
neighboring tetrahedra sharing this $F$ as a common face. 
Recall that $\vn_F$ is the fixed unit vector normal to the face $F$, let $K_+$ 
be the tetrahedron with $\vn_F$ as its inward normal, and $K_-$ with $\vn_F$ as its 
outward normal (see Figure \ref{fig:geom-patch}). Let 
\begin{equation}
\label{eq:rec-wt}
\g^{-}_{F} =\dfrac{\mu^{-1/2}_{K_+}} 
{\mu^{-1/2}_{K_-} + \mu^{-1/2}_{K_+}}
\,\mbox{ and }\,
\k^-_{F} = \dfrac{\b^{1/2}_{K_+}} {\b^{1/2}_{K_-} + \b^{1/2}_{K_+}},
\end{equation}
and $\g^+_{F} = 1- \g^-_{F}$, $\k^+_{F} = 1- \k^-_{F}$. If $F$ is a boundary face 
with its neighboring tetrahedron $K_-$, we set $\g^+_{F} =\k^+_{F}= 1$ and 
$\g^-_{F} = \k^-_{F} = 0$.

The local averages $\bsig_{h,F}$ and $\btau_{h,F}$ on face $F$ are chosen using the 
weights above:
\begin{equation}
\label{eq:rec-favg}
\bsig_{h,F} = \g^-_{F} \bsig_{h,K_-} + 
\g^+_{F}\bsig_{h,K_+}
\,\mbox{ and }\,
\btau_{h,F} = \k^-_{F} \btau_{h,F_-} 
+ \k^+_{F}\btau_{h,F_+}
\end{equation}

respectively, where $\bsig_{h,K_{\pm}} = \bsig_h\at{K_{\pm}}$, and
$\btau_{h,F_{\pm}}(\vx) = \lim\limits_{\e \to 0^{\pm}}\btau_h(\vx + \e \vn_F)$ on 
face $F$. The notation discrepancy in above construction is due to the fact, which 
is mentioned earlier in previous subsection, that 
$\btau_h$'s normal component on each face is a linear polynomial, yet $\bsig_h$ is a 
constant vector on a fixed $K$. 

Now we construct the recovered quantities $\bsig_h^*$ and 
$\btau_h^*$ from the above local averages of $\bsig_h$ and $\btau_h$ as follows:
\begin{equation}
\label{eq:rec-local}
\begin{gathered}
\bsig^*_h(\vx) = \sum_{e\in \cE_h} 
\left\{\sum_{F \subset \wh{\om}_{e,F}}
\Bigg(\frac{1}{|\wh{\om}_{e,F}|}\int_{F} 
(\bsig_{h,F}\cdot \vt_e)\,dS\Bigg) \right\}
\,\bvphi_e(\vx) 
\\
\text{ and }\;
\btau^*_h(\vx) = \sum_{F\in \cF_h} \sum_{\vz \in \cN_h(F)}
 \Bigg(\frac{1}{|F|}\int_{F} (\btau_{h,F}\cdot \vn_F)\lam_{\vz}\,dS\Bigg) 
\,\bpsi_{F,\vz}(\vx) ,
\end{gathered}
\end{equation}
where $\bvphi_e\in \ND_0$ and $\bpsi_{F,\vz}\in \BDM_1$ 
are the nodal basis functions associated with
the respective edge $e$ and vertex $\vz$ on face $F$
(see~\eqref{eq:basis-nd0} and~\eqref{eq:basis-bdm}). 

By the construction of the basis functions in \eqref{eq:basis-nd0} and 
\eqref{eq:basis-bdm}, we can see $\bsig^*_h\in \vHcrl$ and $\btau^*_h\in \vHdiv$, 
respectively. The degrees of freedom of $\btau^*_h$ are the weighted averages of 
$\btau_h:= \b\vu_h$ on a face patch. The degrees of freedom of $\bsig^*_h$ is the 
weighted averages of $\bsig_h:= \mu^{-1}\curlt\vu_h$ on selected interior faces in 
an edge patch.

Now we 
may define the local error indicator ${\eta}_K^2 = {\eta}_{K,\perp}^2 + 
{\eta}_{K,0}^2 + {\eta}_{K,R}^2$ based on these averages plus the 
recovery-type element residual:
\begin{equation}
\label{eq:est-locavg}
\begin{aligned}
&{\eta}_{K,\perp} = \norm{\mu^{1/2} 
{\bsig}_h^* - \mu^{-1/2} \curlt \vu_h 
}_{\vL^2(K)}, \quad {\eta}_{\perp}^2 = \sum_{K\in \cT_h} 
{\eta}_{K,\perp}^2, 
\\
&{\eta}_{K,0} = 
\norm{\b^{-1/2}  \btau^*_h - \b^{1/2} \vu_h }_{\vL^2(K)},
\quad {\eta}_{0} ^2 = \sum_{K\in \cT_h} {\eta}_{K,0}^2,
\\
&
{\eta}_{K,R} =
\mu^{1/2}_K h_K\norm{\vf - \b \vu_h - \curlt \bsig^*_h}_{\vL^2(K)}, 
\quad\text{ and }  {\eta}_{R}^2 = \sum_{K\in \cT_h} {\eta}_{K,R}^2.
\end{aligned}
\end{equation}
The global error 
estimator is defined by 
${\eta}^2= {\eta}_{\perp}^2 + {\eta}_{0}^2 + {\eta}_{R}^2$. 

\section{Reliability and Efficiency Bounds}

This section studies the reliability and efficiency of the estimators defined 
in the previous section. The efficiency bound of the local indicator is established 
in section 4.4. To prove the reliability bound of 
the global estimator, we need two tools: (1) a Helmholtz decomposition with weighted 
norm estimate section 4.1, for detailed proof under certain assumption please see 
Appendix A) that splits the error into two parts, and (2) a modified 
Cl\'{e}ment-type  interpolation (section 4.2). 
Under the assumption of a robust weighted Helmholtz decomposition exists, two 
quasi-monotonicity assumptions on the distribution of the 
coefficients, the reliability bound is obtained in section 4.3, and it is 
uniform with respect to the jumps of the coefficients. 

\subsection{Helmholtz Decomposition}
For any vector in $\vHcrlz$, there exists an orthogonal 
decomposition with respect to the bilinear form 
$\cA(\cdot,\cdot)$ (see \cite{Costabel-Dauge-Nicaise, Fernandes-Gilardi}).
This weighted splitting was used in \cite{Beck-Hiptmair-Hoppe-Wohlmuth} for 
continuously differentiable $\mu$ and $\b$ to prove 
the reliability bound of a residual-based a posteriori error estimator. Here we 
first present the robust weighted splitting result as an assumption (Assumption 
\ref{asp:dcp-general}), then in Appendix \ref{appendix:decomp} 
we show the proof of a bound independent of the 
coefficient jump ratio, under certain assumptions about the geometries and the 
relations between coefficients.

Define the $\vX(\Om,\b)$ as the space of curl-integrable functions intersecting 
weighted div-integrable vector fields, and $\vPH{s}$ as the space of 
\emph{piecewisely continuous} vector fields on each subdomain: 
\begin{equation}
\label{eq:sps-decomp}
\begin{aligned}
&\vX(\Om,\b) = \vHdivw{\b}\cap \vHcrlz,
\\[1mm]
&\text{and } \vPH{s} = 
\{\vv \in \vLt: \vv\at{\Om_j} \in \vHs{s}{\Om_j}, \, j = 1,\ldots,m \}.
\end{aligned}
\end{equation}

\begin{assumption}[Weighted Helmholtz decomposition]
\label{asp:dcp-general}
We assume that for any $\vv\in \vHcrlz$, there 
exist $\psi \in \Hoz$ and $\vw \in \vPH{1}\cap \vX(\Om,\b)$ such that the 
following decomposition holds
\begin{equation}
\label{eq:dcp}
\vv = \vw + \nab \psi.
\end{equation}
Moreover, the following estimate holds:
\begin{equation}
\label{eq:dcp-estimate}
\sum^m_{j=1}\norm{\mu^{-1/2}\nab \vw}_{\vL^2(\Om_j)} 
+ \norm{\b^{1/2}\vw}_{\vLt}
+ \norm{\b^{1/2}\nab \psi}_{\vLt} \leq C\enorm{\vv} .
\end{equation}
\end{assumption}

\subsection{Cl\'{e}ment-type N\'{e}d\'{e}lec Interpolation}  

Weighted Cl\'{e}ment-type interpolation operators for nodal 
Lagrange elements are studied in \cite{Bernardi-Verfurth,Petzoldt02}. Stability and 
approximation properties of this type of operators are often 
used in proving a robust reliability bound for a posteriori error estimators. 
For N\'{e}d\'{e}lec elements, the standard unweighted quasi-interpolations for 
N\'{e}d\'{e}lec elements are studied in 
\cite{Beck-Hiptmair-Hoppe-Wohlmuth,Nicaise07,Schoberl07}. In 
\cite{Beck-Hiptmair-Hoppe-Wohlmuth}, the author defines the edge degrees of 
freedom by averaging on a certain face where that edge lies, which is similar 
to the construction of the Scott-Zhang interpolation operators. 
In \cite{Nicaise07}, the averaging is performing on the edge patch consisting 
of two triangles in 2D. Following the idea of averaging on the weighted vertex patch 
in \cite{Bernardi-Verfurth,Petzoldt02}, and extending the averaging technique on 
edge patch in \cite{Nicaise07} to the three 
dimensional case, we construct a weighted Cl\'{e}ment-type N\'{e}d\'{e}lec 
interpolation operator from $\vHcrlz$ to $\ND_0$. If the vector field to be 
interpolated has $\vPH{1}$ regularity, then the approximation and 
stability properties of the interpolant are proved to be robust under the weighted 
norm, with the assumption that the coefficient is quasi-monotone in Assumption 
\ref{asp:mnt}.

First we define the standard N\'{e}d\'{e}lec interpolation in any $K\in 
\cT_h$.  To make this interpolant well-defined and bounded, 
we need to restrict that the vector field to be interpolated on each element 
$K$ lies in the space $\vH^{1/2+\d}(K)$ for some $ \d>0 $, with its curl in 
$\vL^p(K)$ for some $p>2$ (see \cite{Monk} Lemma 5.38).

\begin{definition}[N\'{e}d\'{e}lec interpolation]
\label{def:interp-nedelec}
For any $\vv\at{K} \in \vH^{1/2+\d}(K)$ with $\curlt \vv\at{K}\in \vL^p(K)$, 
define the interpolation 
operator $\sprod_h: \vH^{1/2+\d}(K) \to \ND_0$ on each element $ 
K\in \cT_h$ as follows: 
\[
\sprod_h \vv \at{K}= \sum_{e\in \cE_h(K)} \a_{e}(\vv) \,\bvphi_e, \quad 
\text{with } \;\a_{e}(\vv) = \frac{1}{|e|}\int_e \vv\cdot \vt_e \,ds.
\]
\end{definition}

\begin{definition}[Weighted Cl\'{e}ment-type N\'{e}d\'{e}lec interpolation]
\label{def:interp-weighted}
For any $\vv\in \vHcrlz$, such that 
$\vv\at{\om_{K,e}} \in \vH(\curl;\om_{K,e}) \cap {P\!\vH}^{1}(\om_{K,e},\sP)$, 
define the weighted quasi-interpolation operator $ \wt{\sprod}_h: \vLt \to 
\ND_0 $ on each element $ K\in \cT_h $ as follows: 
\[
\wt{\sprod}_h \vv\at{K}= \sum_{e\in \cE_h(K)} \wt{\a}_{e}(\vv)\,\bvphi_e ,
\quad \text{with }\;
\wt{\a}_{e}(\vv) =\left(\frac{1}{|\wt{\om}_e|}\int_{\wt{\om}_e} 
\vv\,d\vx\right)\cdot \vt_e
\]
if 
$e$ is an interior edge, i.e., the 1-dimensional Lebesgue measure 
$\mathrm{meas}_1 (e\cap\p\Om) = 0$. 
If $e\in \cE_h(\p \Om)$, then $\wt{\a}_{e}(\vv) = 0$.
\end{definition}

To establish the stability and approximation bounds for this interpolation 
uniform with respect to $\mu^{-1}$, a quasi-monotonicity assumption is needed 
on the distribution of the coefficients associated with each edge patch 
$\omega_e$ in three dimensions, which is similar to those of 
\cite{Bernardi-Verfurth,Petzoldt02} associated with each vertex patch in two 
dimensions. The quasi-monotonicity, in layman's terms, can be phrased as ``for every 
element in an edge patch, there exists a simply-connected element path leading to 
the element where the coefficient achieves the maximum (or minimum) on this patch''. 
The following assumption is stated in a mathematically rigorous way to convey above 
idea.

\begin{assumption}[Quasi-monotonicity of the $\mu^{-1}$ in an edge patch]
\label{asp:mnt}
For each edge $e\in\cE_h$, if $e$ is an interior edge, 
for every $K\subset \om_{e}$, and every $K'\subset \wt{\om}_{e}$, 
{\em (i)} assume that there exist a collection of elements 
$\displaystyle\om'_e= \mcup_{i=1}^{l(K,e)} K_i\subset \om_e$ 
with $K_{l(K,e)}\subset \wt{\om}_e$, 
such that $K_{i}$ shares a face with $K_{i-1}$, and 
that $\mu_{K_{i-1}}^{-1}\le \mu_{K_i}^{-1}$ for all $i = 1,\dots, l(K,e)$, 
where $K_0 = K$.
If $e$ is a boundary edge, for every $K\subset \om_{e}\backslash \wt{\om}_e$, 
and every $K'\subset \wt{\om}_{e}$, {\em (ii)} assume that {\em (i)} holds, 
and the 
2-dimensional Lebesgue measure $\mathrm{meas}_2(\p\wt{\om}_e\cap \p\Om) > 0 $.
\end{assumption}

The assumption is phrase using $\wt{\om}_{e}$, the assumption remains the same if we 
switch $\wt{\om}_{e}$ to $\wh{\om}_{e}$, and reverse the direction of the 
inequalities. 
 
If Assumption \ref{asp:mnt} is satisfied, the extended 
$\mu$-weighted patch for an element $K$ is denoted as 
\begin{equation}
\label{eq:patch-Kewt}
\wt{\om}_{K,e} = K \mcup_{e\in \cE_h(K)}  
\om'_e \mcup_{e\in \cE_h(K)} \wt{\om}_e.
\end{equation}

\begin{remark}
Assumption~{\em \ref{asp:mnt} (i)} is weaker than the extension of 
the quasi-monotonicity assumption in {\em\cite{Bernardi-Verfurth}}, 
and is the equivalent to the extension of the quasi-monotonicity assumption in 
{\em\cite{Petzoldt02}} from the vertex patch in two dimensions to the edge 
patch in three dimensions. 
Notice if Assumption {\em \ref{asp:mnt}} is met, then $\wt{\om}_e$ 
is a simply connected Lipschitz polyhedron for any interior edge $e$.
\end{remark}

\begin{figure}[H]
\begin{center}
\subfloat[Quasi-monotone in the senses of the extension of 
\cite{Bernardi-Verfurth}, the extension of \cite{Petzoldt02}, and Assumption 
\ref{asp:mnt}.]
{\includegraphics[scale=1]{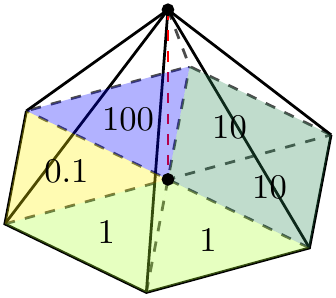}
\label{fig:geom-1}
}
\qquad
\subfloat[Quasi-monotone in the senses of the extension of 
\cite{Petzoldt02} and Assumption \ref{asp:mnt}, not in the 
extension of \cite{Bernardi-Verfurth}.]
{
\includegraphics[scale=1]{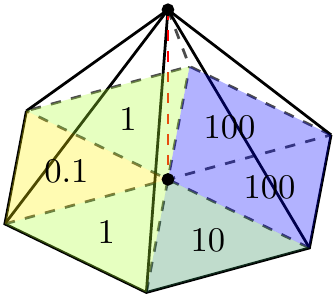}
\label{fig:geom-2}
}
\qquad
\subfloat[Not quasi-monotone in any sense.]
{
\includegraphics[scale=1]{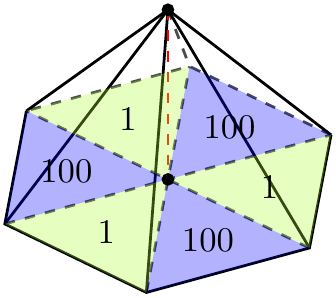}
\label{fig:geom-3}
}
\end{center}

\caption{Different scenarios of the coefficient distribution for $\mu^{-1}$ 
for an interior edge patch $\om_e$, where the edge $e$ is marked as red dotted 
vertical edge in each figure. The tetrahedra whose bases are marked using 
blue color in (a) and (b) consist the $\wt{\om}_e$ for this edge patch.}
\label{fig:geom-edge}
\end{figure}

The illustrations in Figure \ref{fig:geom-edge} show the difference and 
similarity between the Assumption \ref{asp:mnt} and the extension 
to those in \cite{Bernardi-Verfurth,Petzoldt02}. 
In Figure \ref{fig:geom-1}, for any two tetrahedra in the edge patch $\om_e$, 
there always exists a monotone path connecting 
these made of tetrahedra, along which one tetrahedron shares one face with the 
next tetrahedron in this path. In Figure \ref{fig:geom-2}, along the path from 
any tetrahedron in this patch to the one with the biggest coefficients $\mu^{-1}$, 
the coefficients are monotone. In Figure \ref{fig:geom-3}, the coefficient 
distribution of the checkerboard type is not quasi-monotone in any sense, and a 
Cl\'{e}ment-type interpolation cannot achieve a robust bound (see 
\cite{Petzoldt02,Xu91}), if the edge of interest is an interior edge of the 
triangulation.

\begin{figure}[h]
\begin{center}

\begingroup
\captionsetup[subfigure]{width=0.7\textwidth, %
justification=raggedright,margin={-1cm,-0.1cm}}
\subfloat[Quasi-monotone in the sense of Assumption \ref{asp:mnt}, 
the extension of \cite{Petzoldt02}, and the extension of 
\cite{Bernardi-Verfurth}.]
{\includegraphics[scale=1]{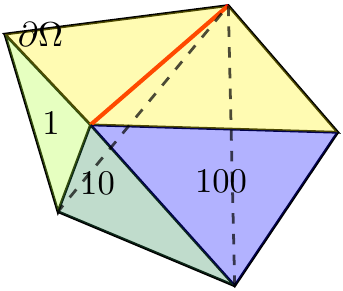}
\label{fig:geom-6}
}
\qquad\qquad
\subfloat[Quasi-monotone in the sense of Assumption \ref{asp:mnt} 
and the extension of \cite{Petzoldt02}, 
not in the sense of the extension of \cite{Bernardi-Verfurth}.]
{\includegraphics[scale=1]{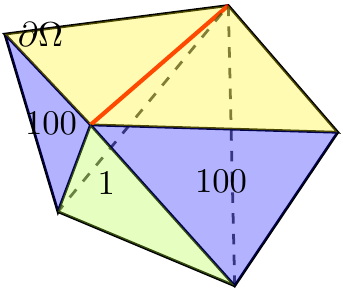}
\label{fig:geom-7}
}

\subfloat[Quasi-monotone for the edge patch in the sense of Assumption 
\ref{asp:mnt}, not quasi-monotone for 
vertex patch for the black dotted vertex.]
{\includegraphics[scale=1]{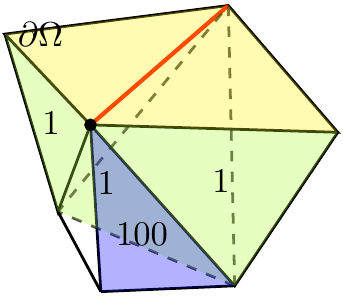}
\label{fig:geom-8}
}
\qquad\qquad
\subfloat[Not quasi-monotone in any sense.]
{\includegraphics[scale=1]{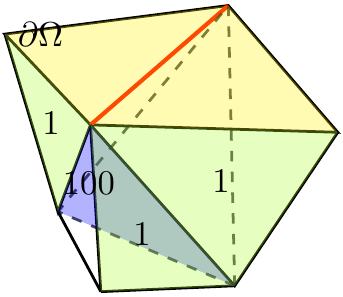}
\label{fig:geom-9}
}
\end{center}
\endgroup
\caption{Different scenarios of the coefficient distribution for $\mu^{-1}$ 
for a boundary edge patch $\om_{e}$, where $e$ is marked red, the boundary 
faces are marked yellow, and the coefficient in each tetrahedron is marked on 
its front faces towards the viewer.}
\label{fig:geom-boundary}
\end{figure}

\begin{theorem}[Approximation and stability properties]
\label{thm:interp-weighted}
Under Assumption~{\em \ref{asp:mnt}}, the interpolation operator $ 
\wt{\sprod}_h$ in Definition~{\em \ref{def:interp-weighted}}, 
satisfies the following estimates:
\begin{equation}
\label{eq:interp-estimate}
\begin{gathered}
\norm{\mu^{-1/2} (\vv - \wt{\sprod}_h \vv)}_{\vL^2(K)} 
\leq C\, h_K \norm{\mu^{-1/2} \nab \vv}_{\vL^2(\wt{\om}_{K,e})},
\\[1mm]
\mbox{and }\,  \norm{\mu^{-1/2} \curlt (\vv - \wt{\sprod}_h \vv)}_{\vL^2(K)} 
\leq C \norm{\mu^{-1/2} \nab \vv}_{\vL^2(\wt{\om}_{K,e})}
\end{gathered}
\end{equation}
for all $\vv\in \vHcrlz$ such that 
$\vv\at{\om_{K,e}} \in \vH(\curl;\om_{K,e}) \cap {P\!\vH}^{1}(\om_{K,e},\sP)$, 
where the Jacobian $\nab \vv$ is defined piecewisely by 
$\nab \vv\at{K} := \nab (\vv\at{K})$.
\end{theorem}

\begin{proof}
To establish the inequalities in \eqref{eq:interp-estimate}, let 
$\bar{\vv}_K$ and $\bar{\vv}_{\wt{\om}_e}$ be the averages of $\vv$ over $K$ and 
$\wt{\om}_e$ respectively, 
i.e.,  $\bar{\vv}_K = |K|^{-1} \int_{K} \vv \,d\vx$, and 
$\bar{\vv}_{\wt{\om}_e} = |\wt{\om}_e|^{-1} \int_{\wt{\om}_e} \vv \,d\vx$ for 
an interior edge $e$. Let $\bar{\vv}_{\wt{\om}_e} = \bm{0}$ if $e$ is a boundary 
edge. 

If $e$ is an interior edge, we have the following standard approximation 
property (also known as Poincar\'{e} inequality) 
thanks to the shape regularity of the triangulation $\cT_h$, simply-connectedness of 
$\wt{\om}_e$ for an interior edge $e$ from Assumption \ref{asp:mnt}, and 
$\vv \in {P\!\vH}^{1}(\om_{K,e},\sP)$:
\begin{equation}
\label{eq:interp-approx}
\begin{aligned}
\norm{\vv - \bar{\vv}_K}_{\vL^2(K)} \leq C h_K \norm{\nab \vv}_{\vL^2(K)},
\;\text{ and } \;
\norm{\vv - \bar{\vv}_{\wt{\om}_e}}_{\vL^2(\wt{\om}_e)} 
\leq C h_e \norm{\nab \vv}_{\vL^2(\wt{\om}_e)}.
\end{aligned}
\end{equation}

If $e$ is a boundary edge, the first inequality in \eqref{eq:interp-approx} 
still holds. To get an equality similar to the second one, the fact that 
$\vv\in \vHcrlz$ and $\vv\at{\om_{K,e}} \in {P\!\vH}^{1}(\om_{K,e},\sP)$ 
implies $\vv\cdot \vt_e\at{\wt{\om}_e} \in H^1(\wt{\om}_e)$, and 
$\vv\cdot \vt_e\at{\p \wt{\om}_e\cap \p \Om} = 0$. The following Friedrichs 
inequality holds (even if $\wt{\om}_e$ is not simply-connected as in the case 
of Figure \ref{fig:geom-7})
\begin{equation}
\label{eq:interp-approx-bd}
\norm{\vv\cdot \vt_e}_{L^2(\wt{\om}_e)} 
\leq C h_e \norm{\nab (\vv\cdot \vt_e)}_{L^2(\wt{\om}_e)} 
\leq C h_e \norm{\nab \vv}_{\vL^2(\wt{\om}_e)} .
\end{equation}

The starting point of the proof is to split the error we want to bound into parts. 
On any $K\in \cT_h$, it follows from the triangle inequality that
\[
\norm{\vv - \wt{\sprod}_h \vv}_{\vL^2(K)} \leq 
\norm{\vv - \bar{\vv}_K}_{\vL^2(K)}
+ \norm{\bar{\vv}_K - \wt{\sprod}_h \vv}_{\vL^2(K)}.
\]
The first term can be estimated using \eqref{eq:interp-approx} first 
inequality. For the second term, since $\wt{\sprod}_h \bar{\vv}_K= 
\bar{\vv}_K$ (with slightly abuse of notation we can extend $\bar{\vv}_K$ to whole 
edge patch by letting it be its value on $K$), we have the following partition on 
the element $K$ 
\[
\bar{\vv}_K - \wt{\sprod}_h \vv = \wt{\sprod}_h (\bar{\vv}_K - \vv) = 
\sum_{e\in \cE_h(K)} (\bar{\vv}_K - \bar{\vv}_{\wt{\om}_e}) 
\cdot \vt_e \,\bvphi_{e}.
\]
Now applying the triangle inequality, and using the fact that  
$\norm{\bvphi_e}_{\vL^2(K)} \leq C |K|^{1/2}$ 
(see the construction of $\bvphi_e$ in~\eqref{eq:basis-nd0}) yield
\begin{equation}
\label{eq:interp-partition}
\begin{aligned}
&\norm{\bar{\vv}_K - \wt{\sprod}_h \vv}_{\vL^2(K)} 
\leq  \sum_{e\in \cE_h(K)}\norm{(\bar{\vv}_K - \bar{\vv}_{\wt{\om}_e}) \cdot 
\vt_e \,\bvphi_{e}}_{\vL^2(K)} 
\\
=& 
\sum_{e\in \cE_h(K)} \abs{(\bar{\vv}_K - \bar{\vv}_{\wt{\om}_e}) \cdot \vt_e} 
\cdot \norm{\bvphi_{e}}_{\vL^2(K)} \leq 
C |K|^{1/2} 
\sum_{e\in \cE_h(K)}\abs{(\bar{\vv}_K - \bar{\vv}_{\wt{\om}_e}) \cdot \vt_e}.
\end{aligned}
\end{equation}

To establish the estimate for 
$\abs{(\bar{\vv}_K - \bar{\vv}_{\wt{\om}_e}) \cdot \vt_e}$ for each edge, we 
consider three cases. The first case is that when $K\subset \wt{\om}_e$, using 
the triangle inequality, the estimates in \eqref{eq:interp-approx} and 
\eqref{eq:interp-approx-bd} gives
\[
\begin{aligned}
&|K|^{1/2}\abs{(\bar{\vv}_K - \bar{\vv}_{\wt{\om}_e}) \cdot \vt_e}
= \norm{(\bar{\vv}_K - \bar{\vv}_{\wt{\om}_e}) \cdot \vt_e}_{L^2(K)}
\\[1mm]
\leq &  \norm{(\bar{\vv}_K - \vv)\cdot \vt_e}_{L^2(K)} 
+  \norm{(\vv - \bar{\vv}_{\wt{\om}_e})\cdot \vt_e}_{L^2(K)} 
\\
\leq & \norm{\bar{\vv}_K - \vv}_{\vL^2(K)} 
+ \norm{(\vv - \bar{\vv}_{\wt{\om}_e}) \cdot \vt_e}_{L^2(\wt{\om}_e)} 
\leq C  \frac{h_K}{\mu^{-1/2}_K} \norm{\mu^{-1/2} \nab \vv}_{\vL^2(\wt{\om}_e)}.
\end{aligned}
\]
Here the term in front of the last inequality is treated as
$\norm{(\vv - \bar{\vv}_{\wt{\om}_e}) \cdot \vt_e}_{L^2(\wt{\om}_e)}
\leq \norm{\vv - \bar{\vv}_{\wt{\om}_e}}_{\vL^2(\wt{\om}_e)}$ for an interior 
edge, and 
$\norm{(\vv - \bar{\vv}_{\wt{\om}_e}) \cdot \vt_e}_{L^2(\wt{\om}_e)} = 
\norm{\vv \cdot \vt_e}_{L^2(\wt{\om}_e)}$ for a boundary edge.

The second case is that when $K\not\subset \wt{\om}_e$, yet $K$ is adjacent to 
$\wt{\om}_e$, and we denote the face they share as $\p K\cap \p \wt{\om}_e = 
F$.  The fact that the tangential component of $\vv$ along the edge $e$ is 
continuous across the face $F$, and $\vv \in {P\!\vH}^{1}(\om_{K,e},\sP)$ implies 
that $\vv\cdot \vt_e \in H^1(K\cup \wt{\om}_e\cup F)$ 
(e.g. see \cite{Monk} Lemma 5.3). 
To establish the estimate, we need a standard trace inequality for 
$p\in H^1(K\cup \wt{\om}_e\cup F)$ (e.g. see \cite{Verfurth99} Lemma 3.2):
\begin{equation}
\label{eq:interp-trace}
\norm{p}_{L^2(F)} \leq C\left\{
h^{-1/2}_F \norm{p}_{L^2(K')} + h^{1/2}_F \norm{\nab p}_{L^2(K')}
\right\},
\end{equation}
where $K'$ can be either the element of interest $K$, or the element $\wt{K}$ 
as a subset of $\wt{\om}_e$ which is adjacent to $K$. 

Now it follows from the 
triangle inequality and shape regularity of the triangulation that
\begin{equation}
\label{eq:interp-keyterm}
\begin{aligned}
& |K|^{1/2} \abs{(\bar{\vv}_K - \bar{\vv}_{\wt{\om}_e}) \cdot \vt_e} 
= \frac{|K|^{1/2}}{|F|^{1/2}} 
\norm{(\bar{\vv}_K - \bar{\vv}_{\wt{\om}_e}) \cdot \vt_e}_{L^2(F)}
\\
\leq &\, C h_K^{1/2} \norm{(\bar{\vv}_K - \vv) \cdot \vt_e}_{L^2(F)}
+ C h_K^{1/2} \norm{(\vv - \bar{\vv}_{\wt{\om}_e}) \cdot \vt_e}_{L^2(F)}.
\end{aligned}
\end{equation}

The first term in the \eqref{eq:interp-keyterm} can be estimated using 
\eqref{eq:interp-trace} and then \eqref{eq:interp-approx} 
\[
\begin{aligned}
& h_K^{1/2} \norm{(\bar{\vv}_K - \vv) \cdot \vt_e}_{L^2(F)} 
\\[1mm]
\leq & \, C h_K^{1/2} \left\{
h^{-1/2}_F \norm{(\bar{\vv}_K - \vv) \cdot \vt_e}_{L^2(K)} 
+ h^{1/2}_F \norm{\nab (\vv\cdot \vt_e)}_{L^2(K)}
\right\}
\\[1mm]
\leq & \, C \left\{\norm{\bar{\vv}_K - \vv}_{\vL^2(K)} 
+ h_K \norm{\nab \vv}_{\vL^2(K)}
\right\} \leq C  h_K \norm{\nab \vv}_{\vL^2(K)}.
\end{aligned}
\]
For the second term 
$h_K^{1/2} \norm{(\vv - \bar{\vv}_{\wt{\om}_e}) \cdot \vt_e}_{L^2(F)}$ 
in \eqref{eq:interp-keyterm}, using the same argument yields a similar 
estimate, except passing the trace inequality from the face $F$ to the element 
$\wt{K} \subset \wt{\om}_e$ this time:
\[
h_K^{1/2} \norm{(\vv - \bar{\vv}_{\wt{\om}_e}) \cdot \vt_e}_{L^2(F)}
\leq C h_{\wt{K}} \norm{\nab \vv}_{\vL^2(\wt{\om}_e)}.
\]
Combining the two inequalities obtained above gives the following estimate
for any $e\in \cE_h(K)$ thanks to $\mu_K^{-1} \leq \mu_{\wt{\om}_e}^{-1}$:
\begin{equation}
\label{eq:interp-avgdiff}
|K|^{1/2} \abs{(\bar{\vv}_K - \bar{\vv}_{\wt{\om}_e}) \cdot \vt_e}
\leq 
C \frac{h_K}{\mu^{-1/2}_K} \norm{\mu^{-1/2} \nab \vv}_{\vL^2(\wt{\om}_e)}.
\end{equation}

The third case is that when $K\not\subset \wt{\om}_e$, nor does $K$ share a 
face with $\wt{\om}_e$. By Assumption~\ref{asp:mnt} 
there is a simply connected patch consisting of $K_1,\dots, K_{l(K,e)-1}$ 
along which the $\mu^{-1}$ is monotone. Separating the term of interest by 
triangle inequality:
\[
\abs{(\bar{\vv}_K - \bar{\vv}_{\wt{\om}_e}) \cdot \vt_e} 
\leq \abs{(\bar{\vv}_K - \bar{\vv}_{K_1}) \cdot \vt_e} 
+ \dots  
+ \abs{(\bar{\vv}_{K_{l(K,e)-1}} - \bar{\vv}_{\wt{\om}_e}) \cdot \vt_e},
\]
then each of the above terms can be proved yielding the same form of estimate 
in~\eqref{eq:interp-avgdiff} by the same argument. This result, together with 
the representation of $\norm{\bar{\vv}_K - \wt{\sprod}_h \vv}_{\vL^2(K)}$ in 
\eqref{eq:interp-partition}, implies the first estimate in 
\eqref{eq:interp-estimate}.

For the second estimate in \eqref{eq:interp-estimate}, using the inverse 
inequality, the triangle inequality, and $\wt{\sprod}_h \bar{\vv}_K= 
\bar{\vv}_K$ again, we have that
\[
\begin{aligned}
\norm{\curlt \wt{\sprod}_h \vv}_{\vL^2(K)} 
&= \norm{\curlt \wt{\sprod}_h (\vv -\bar{\vv}_K )}_{\vL^2(K)}
\leq C \,h_K^{-1} \norm{\wt{\sprod}_h (\vv - \bar{\vv}_K)}_{\vL^2(K)}
\\[1mm]
&\leq C \,h_K^{-1} \norm{\wt{\sprod}_h \vv - \vv}_{\vL^2(K)}
+ C \,h_K^{-1} \norm{\vv - \bar{\vv}_K}_{\vL^2(K)},
\end{aligned}
\]
which, together with the first estimate in \eqref{eq:interp-estimate} and 
\eqref{eq:interp-approx}, implies the second estimate. This completes the proof 
of the theorem.
\end{proof}

\begin{assumption}[Quasi-monotonicity of the $\b$ in a vertex patch]
\label{asp:mntb}
For any vertex $\vz\in \cN_h$, assume that the 
$\b$ satisfies the vertex patch quasi-monotonicity condition in 
{\em\cite{Petzoldt02}}: 
if $\vz$ is an interior vertex, for any $K\subset \om_{\vz}$, 
and $K'\subset \wt{\om}_{\vz}$, {\em (i)} there exist a 
collection of elements 
$\displaystyle\om'_{\vz}= \mcup_{i=1}^{l(K,\vz)} K_i\subset \om_{\vz}$ 
with $K_{l(K,\vz)}\subset \wt{\om}_{\vz}$, such that 
$K_{i}$ shares a face with $K_{i-1}$ and that $\b_{K_{i-1}}\le \b_{K_i}$ for 
all $i = 1,\dots, l(K,\vz)$, where $K_0 = K$ and. If $\vz$ is a vertex on the 
boundary, for every $K\subset \om_{\vz}\backslash \wt{\om}_{\vz}$, and  
every $K'\subset \wt{\om}_{\vz}$, {\em (ii)} assume that {\em (i)} holds, and 
the 2-dimensional Lebesgue measure 
$\mathrm{meas}_2(\p\wt{\om}_{\vz}\cap \p\Om) > 0 $.
\end{assumption}

If Assumption \ref{asp:mntb} is satisfied, the extended 
$\b$-weighted patch for an element $K$ is denoted as
\[
\wt{\om}_{K,\vz} = K \mcup_{\vz\in \cN_h(K)}  \om'_{\vz} 
\mcup_{\vz\in \cN_h(K)} \wt{\om}_{\vz},
\]

For the $\b$ which satisfies the vertex patch quasi-monotonicity in Assumption 
\ref{asp:mntb}, the robust  
Cl\'{e}ment-type interpolation for the linear Lagrange elements results are 
already established in \cite{Bernardi-Verfurth,Petzoldt02}. 
In the three dimensional setting, one reason to study the 
Cl\'{e}ment-type interpolation is that the standard linear 
Lagrange nodal interpolant may not be bounded. Unless extra regularity is 
assumed (e.g. the function to be interpolated is in $H^{3/2+\e}(\Om)$, see 
\cite{Monk}), the degrees of freedom for the Lagrange nodal interpolant may not 
be well-defined because $\Ho$ is not 
continuously embedded into the continuous function space. 

For any $\psi\in \Hoz$, let $\psi_h$ be the weighted Cl\'{e}ment-type 
interpolant of $\psi$ defined in \cite{Petzoldt02} associated with the 
coefficient $\b$. Under Assumption \ref{asp:mntb}, the $\psi_h$ has 
the following properties:
\begin{equation}
\label{eq:interp-estimate-h1}
\begin{aligned}
&\norm{\b^{1/2}(\psi-\psi_h(\vz))\lam_{\vz} }_{L^2(K)} 
\leq c_1 \,h_K \norm{\b^{1/2}\nab \psi}_{\vL^2(\wt{\om}_{K,\vz})},
\\[1mm]
\text{ and } & \norm{\b^{1/2}\nab(\psi-\psi_h) }_{\vL^2(K)} 
\leq c_2 \norm{\b^{1/2}\nab \psi}_{\vL^2(\wt{\om}_{K,\vz})},
\end{aligned}
\end{equation}
for any vertex $ \vz \in \cN_h(K)$. 

\begin{figure}[h]
\begin{center}

\begingroup
\captionsetup[subfigure]{labelsep=quad,margin={-2cm,0.1cm}}
\subfloat[Quasi-monotone in the sense of Assumption \ref{asp:mnt}, 
not in Assumption \ref{asp:mntb}.]
{\includegraphics[scale=1]{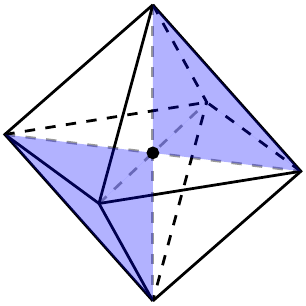}
\label{fig:geom-4}
}
\qquad \qquad \qquad
\subfloat[Not quasi-monotone in any sense.]
{\includegraphics[scale=1]{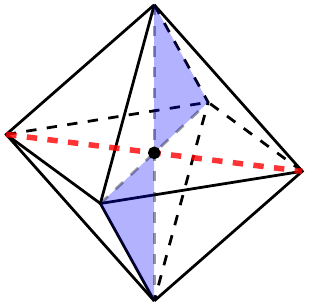}
\label{fig:geom-5}
}
\end{center}

\endgroup
\caption{Different scenarios of the coefficient distribution for $\b$ 
for an interior vertex patch $\om_{\vz}$. (a) $\b=100$ in the tetrahedra whose  
faces are marked blue, $\b=1$ in the rest tetrahedra in this patch. The coefficient 
distrubtion is quasi-monotone for all interior edges within this patch.
(b) $\b=100$ in the four tetrahedra sharing the blue faces,  
$\b=1$ in the rest tetrahedra in this patch. The quasi-monotonicity is violated for 
the red edges.}
\label{fig:geom-vertex}
\end{figure}

\begin{remark}
Assumption {\em \ref{asp:mnt}} does not exclude the case when 
$\ol{\Om} = \ol{\Om}_1\mcup \ol{\Om}_2$, and $\Om_1$ is a Lipschitz polyhedron 
touching the boundary at one vertex $V_1$ only, with $\mu^{-1}\at{\Om_1} =100$, 
and $\mu^{-1}\at{\Om_2} =1$. Assumption {\em \ref{asp:mntb}} 
prohibits the existence of this scenario. 
In this scenario, a robust Cl\'{e}ment-type 
interpolation cannot be achieved for nodal Lagrange elements. 
However, Assumption {\em \ref{asp:mnt}} allows this kind of domain, 
in which all the edges on the $\p\Om_1$ connecting the vertex $V_1$ is an 
interior edge of the triangulation. 
A robust Cl\'{e}ment-type interpolation using 
N\'{e}d\'{e}lec elements does exist in this scenario. Please refer to the 
illustration in Figure {\em \ref{fig:geom-8}}.
\end{remark}

\begin{remark}
Assumption {\em \ref{asp:mnt}} which states quasi-monotonicity for 
the edge patch, is weaker than Assumption {\em \ref{asp:mntb}} for 
the vertex patch. The reason is that Assumption 
{\em \ref{asp:mnt}} allows the checkerboard pattern for a vertex 
patch. However, this vertex patch checkerboard pattern is excluded in 
Assumption {\em \ref{asp:mntb}}. 
Please refer to the illustration in Figure {\em\ref{fig:geom-vertex}}. In 
Figure {\em\ref{fig:geom-4}}, the coefficient distribution satisfies Assumption 
{\em \ref{asp:mnt}} for any interior edges within this patch, 
yet Assumption {\em \ref{asp:mntb}} is not met.
\end{remark}

\subsection{Reliability}
Under the assumption on the distributions of the coefficients and the Helmholtz 
decomposition which is stable under the weighted norm, we prove 
the global reliability for the local recovery error estimator $\eta$ defined 
in~\eqref{eq:est-locavg}. 

For any vertex $\vz\in \cN_h\backslash \cN_h(\p\Om)$, 
denote by
\[
F_{\om_{\vz}}:= 
\dfrac{1}{|\om_{\vz}|}\int_{\om_{\vz}}\divv(\vf - \btau^*_h) d\vx
\]
the average of $\divv(\vf - \btau^*_h)$ over the vertex patch $\om_{\vz}$. 
For $\vz\in \cN_h(\p\Om)$, $F_{\om_{\vz}}:=0$. Let 
\[
\begin{gathered}
H = \left(\sum_{K\in\cT_h} \eta_{K,d}^2\right)^{1/2}
\quad\mbox{with}\quad \eta_{K,d}=
\b_K^{-1/2}h_K\norm{\divv(\vf - \btau^*_h)}_{L^2(K)}
\\[1mm]
\text{ and } \; \wt{H} = \left(\sum_{\vz\in \cN_h }\sum_{K\subset\om_{\vz}} 
\b_K^{-1}h_K^2
\norm{\divv(\vf - \btau^*_h)-F_{\om_{\vz}}}_{L^2(K)}^2\right)^{1/2}.
\end{gathered}
\]
The contribution from interior nodes in $\wt{H}$ is a higher order term since 
$\divv(\vf - \btau^*_h) \in \Lt$, and so is the contribution from boundary 
nodes if $\divv(\vf - \btau^*_h) \in L^p(\Om)$ for some $p>2$.
(see \cite{Carstensen-Verfurth}).

\begin{theorem}[Global Reliability of $\eta$]
\label{thm:rel}
Let $\vu$ and $\vu_h$ be the solutions of~\eqref{eq:pb-variational} 
and~\eqref{eq:pb-fem}, respectively. Under 
Assumption {\em\ref{asp:dcp-general}},
{\em\ref{asp:mnt}}, and {\em\ref{asp:mntb}} , 
there exists a constant $C> 0$ 
independent of the jumps of the coefficients such that
\begin{equation}
\label{eq:rel}
\enorm{\vu - \vu_h} \leq C\left(\eta + H\right).
\end{equation}
\end{theorem}

\begin{proof}
Denote the error and the residual by 
\[
\ve = \vu - \vu_h
\quad\mbox{and}\quad
R(\vv) = (\vf,\vv) - (\mu^{-1}\curlt \vu_h, \curlt\vv) 
- (\b \vu_h,\vv),
\]
respectively. It is easy to see that 
 \[
\cA(\ve,\vv) = R(\vv), \,\, \forall \vv\in \vHcrlz
\;\text{ and } \;
R(\vv_h) = 0,\; \forall \vv_h\in \ND_0\cap \vHcrlz.
\]
By Assumption~\ref{asp:dcp-general}, there 
exists a decomposition of the error $\ve\in \vHcrlz$ into the sum of $\psi \in 
\Hoz$ and $\vw \in \vPH{1}\cap \vX(\Om,\b)$ such that
\[ 
\ve =\vw + \nab \psi \quad \text{and} \quad 
\enorm{\ve}^2= R(\ve) =  R(\vw)+ R(\nab \psi).
\]


To bound the curl-free part of the error, 
let $\psi_h$ be the weighted Cl\'{e}ment-type interpolant of $\psi$ defined in 
\cite{Petzoldt02} associated with the coefficient $\b$. It follows from the 
fact that $R(\nab \psi_h) = 0$, integration by parts, 
the Cauchy-Schwarz inequality, the approximation and stability of the 
interpolation \eqref{eq:interp-estimate-h1}, and \eqref{eq:dcp-estimate} 
that 
\[
\begin{aligned}
R(\nab \psi) &= R\big(\nab(\psi -\psi_h)\big) 
= \binprod{\vf-\btau^*_h}{\nab(\psi -\psi_h)} 
+ \binprod{\btau^*_h - \b \vu_h}{\nab(\psi -\psi_h)}
\\[1mm]
&=-\binprod{\divv(\vf-\btau^*_h)}{\psi -\psi_h}
+ \binprod{\btau^*_h - \b \vu_h}{\nab(\psi -\psi_h)}
\\[1mm]
& \leq \sum_{K\in\cT_h}  \left(
 \eta_{K,d} \,h^{-1}_K\norm{\b^{1/2}(\psi -\psi_h)}_{L^2(\om_{K,\vz})}
+ \eta_{K,0}\, \norm{\b^{1/2}\nab(\psi -\psi_h)}_{\vL^2(\om_{K,\vz})}
\right)
\\[1mm]
& \leq C\sum_{K\in\cT_h} \left(\eta_{K,d} + \eta_{K,0}\right)
\norm{\b^{1/2}\nab\psi}_{\vL^2(\om_{K,\vz})}
\leq C \left(H +\eta_0\right) \norm{\b^{1/2}\nab\psi}_{\vLt}
\\[1mm]
&\leq C \left(H +\eta_0\right) \norm{\b^{1/2}\ve}_{\vLt}.
\end{aligned}
\]

To bound $R(\vw)$, let $\vw_h = \wt{\sprod}_h \vw\in \ND_0\cap \vHcrlz$ with 
$\wt{\sprod}_h$ defined in Definition \ref{def:interp-weighted}. Using the fact that 
$R(\vw_h) =0$ and integrating by parts give
\[
\begin{aligned}
& R(\vw)= R(\vw - \vw_h) 
\\[1mm]
&= (\vf - \b \vu_h,\vw - \vw_h) 
- \binprod{\bsig^*_h}{\curlt(\vw - \vw_h)}
+ \binprod{\bsig^*_h - \mu^{-1}\curlt \vu_h}{\curlt(\vw - \vw_h)}
\\[1mm]
&= \binprod{\vf - \b \vu_h - \curlt \bsig^*_h}{\vw - \vw_h} 
+\binprod{\bsig^*_h - \mu^{-1}\curlt \vu_h}{\curlt(\vw - \vw_h)}.
\end{aligned}
\]
Now, by the Cauchy-Schwarz inequality, the second and third inequalities in 
\eqref{eq:interp-estimate} ,
and \eqref{eq:dcp-estimate},
we have
\[
\begin{aligned}
R(\vw)
& \leq \sum_{K\in\cT_h}\left(
\eta_{K,R}\,h^{-1}_K\norm{\mu^{-1/2}(\vw - \vw_h)}_{\vL^2(K)} 
+\eta_{K,\perp} \norm{\mu^{-1/2}\curlt(\vw - \vw_h)}_{\vL^2(K)}
\right)
\\[1mm]
&\leq C\sum_{K\in\cT_h}\left(
 \eta_{K,R} + \eta_{K,\perp}\right)
 \norm{\mu^{-1/2}\nab \vw}_{\vL^2(\om_K)}
\leq  C\left(
 \eta_{R} + \eta_{\perp}\right)\norm{\mu^{-1/2}\nab \vw}_{\vLt}
 \\[1mm]
&\leq C\left(
 \eta_{R} + \eta_{\perp}\right)
\norm{\mu^{-1/2} \curlt \ve}_{\vLt}.
\end{aligned}
\]
Combining the above two inequalities, we have
\[
\enorm{\ve}^2= R(\ve) \leq C\,\big( \eta + H \big)\enorm{\ve},
\]
which implies \eqref{eq:rel}. This completes the proof of the theorem.
\end{proof}

In the remainder of this section, 
we assume that additionally $\divv \vf \in L^p(\Om)$ for some $p>2$, 
then the $H$ in~\eqref{eq:rel} may be replaced by 
$\wt{H}$ which is a higher order term.

\begin{theorem}[Global Reliability of $\eta$]
\label{thm:rel-hot}
Under Assumption {\em\ref{asp:coeff}}, {\em\ref{asp:mnt}},
and {\em\ref{asp:mntb}}, there exists a constant $C> 0$ 
independent of the jumps of the coefficients such that
\begin{equation}
\label{eq:rel-hot}
\enorm{\vu - \vu_h} \leq C\left(\eta + \wt{H}\right).
\end{equation}
\end{theorem}

\begin{proof}
In the proof of \eqref{eq:rel}, 
if furthermore the following orthogonality condition is exploited on the vertex 
patch for the weighted Cl\'{e}ment-type interpolant (e.g. see \cite{Cai09} Section 4)
\[
\binprod{1}{(\psi - \psi_h(\vz))\lam_{\vz}}_{\om_{\vz}} = 0,\quad\forall\,\, 
\vz\in\cN_h\backslash \cN_h(\p \Om),
\]
together with the fact that $F_{\om_{\vz}}=0$ and $\psi_h(\vz)=0$ for $\vz\in 
\cN_h(\p\Om)$, it implies
\[
\begin{aligned}
&\binprod{\vf-\btau^*_h}{\nab(\psi -\psi_h)} 
=  -\sum_{K\in\cT_h}\binprod{\divv(\vf-\btau^*_h)}{\psi -\psi_h}_K
\\
&\qquad= -\sum_{\vz\in \cN_h} \sum_{K\subset\om_{\vz}} 
\binprod{\divv(\vf-\btau^*_h)-F_{\om_{\vz}}}{(\psi -\psi_h(\vz))\lam_{\vz}}_K.
\end{aligned}
\]
Now, a similar argument as in the irrotational part 
proof of \eqref{eq:rel} gives
\[
\binprod{\vf-\btau^*_h}{\nab(\psi -\psi_h)} 
\leq C \,\wt{H}\norm{\b^{1/2}\ve}_{\vLt}.
\]
The rest of the proof for (\ref{eq:rel-hot}) is identical to that of 
\eqref{eq:rel}.
\end{proof}

\subsection{Efficiency}
Even though in \cite{Beck-Hiptmair-Hoppe-Wohlmuth}, the coefficients are assumed to 
be continuous, the proof they used to prove the 
efficiency bound (Section 4 and 5 in \cite{Beck-Hiptmair-Hoppe-Wohlmuth}) carries 
over to piecewise constant coefficients. At the same time, their choice of weight 
yields a robust bound with no dependence on the coefficients. In this subsection, we 
prove the efficiency of the recovery-based estimator \eqref{eq:est-locavg} by 
bounding the recovery-based local error estimator by the residual-based local error 
estimator.

Let $\vf_h$ be the standard $\vL^2$-projection 
onto $\BDM_1$. It is proved in \cite{Beck-Hiptmair-Hoppe-Wohlmuth} 
that there exists a positive constant $C$ such that:
\begin{equation}
\label{eq:eff-residual}
\begin{aligned}
& C h_F \norm{\mu_F^{1/2} 
\jump{(\mu^{-1} \curlt \vu_h)\cross \vn}{F}}_{\vL^2(F)}^2 
 \leq  
\enorm{\vu - \vu_h}_{\om_{F}}^2\!\!
+\!\! \sum_{K\subset\om_{F}}\!\!
\mu_K h_K^2\norm{\vf-\vf_h}_{\vL^2(K)}^2,
\\[2mm]
& C h_F
\norm{\b_F^{-1/2}  \jump{\b \,\vu_h\cdot \vn}{F}}_{L^2(F)}^2
\leq 
\norm{\b^{1/2} (\vu - \vu_h)}_{\vL^2(\om_F)}^2\!\!\!
+\!\!\! \sum_{K\subset\om_F} \!\!
\b_K^{-1} h_K^2\norm{\divv\vf}_{L^2(K)}^2,
\\[2mm]
& \text{and }\;C\,\mu_K^{1/2} h_K
\norm{\vf - \b \vu_h - \curlt (\mu^{-1} \curlt \vu_h)}_{\vL^2(K)}
\leq \enorm{\vu-\vu_h}_K, 
\end{aligned}
\end{equation}
where the coefficients $\mu_F^{-1}$ and $\b_F$ on face $F$ are given by the 
arithmetic averages of $\mu^{-1}$ and $\b$ 
\[
\mu_F^{-1} = (\mu_{K_-}^{-1}+\mu_{K_+}^{-1})/2, 
\text{ and } 
\b_F = \big(\b_{K_-} + \b_{K_+}\big)/2,
\] 
respectively. Next we move on to prove the equivalence.
 
\begin{lemma}[Equivalence of ${\eta}_{K,0}$]
\label{lem:eff-0}
There exists a constant $c > 0$ independent of the 
jumps of the coefficients such that for any $K \in \cT_h$:
\begin{equation}
\label{eq:eff-0}
c\,{\eta}_{K,0} \leq  
\sum_{F\subset \p K} h_F^{1/2}
\norm{\b_F^{-1/2} \jump{\b \vu_h\cdot \vn}{F}}_{L^2(F)}.
\end{equation}
\end{lemma}

\begin{proof}
it suffices to show that ${\eta}_{K,0}$ can be bounded by the summation of the 
residual-based estimator within the local face patch. 

For any interior element $K$, we first use a partition of unity to bound the 
estimator ${\eta}_{K,0}$ by the fact that $\ND_0(K)\subset \BDM_1(K)$.
The difference of the weighted average $\btau^*_h$ and $\btau_h$ is
\[
\begin{aligned}
(\btau^*_h - \btau_h)\at{K} 
&= \sum_{F\subset \p K\backslash \p \Om} \sum_{\vz \in \cN_h(F)}
\Bigg(\frac{1}{|F|}\int_{F} (\btau_{h,F} - \btau_{h,K})\cdot \vn_F 
\lam_{\vz}\,dS\Bigg) 
\,\bpsi_{F,\vz}
\\
&= \sum_{F\subset \p K\backslash \p \Om}\, 
\sum_{\vz\subset \cN_h(F)} \dfrac{1 - \k_{F}^{K}} {|F|}
\left(\int_F\jump{\b \vu_h\cdot \vn}{F}\lam_{\vz}\,dS\right) \,\bpsi_{F,\vz}.
\end{aligned}
\]
Recalling from \eqref{eq:rec-wt} that on each face $F$ of element $K$, 
$\k_{F}^{K} = \b^{1/2}_{K'}/\big( {\b^{1/2}_{K} + \b^{1/2}_{K'}}\big)$, where $K'$ 
is 
the neighboring element sharing this fixed face $F$ with $K$. Since
\[
\norm{\lam_{\vz}}_{L^2(F)}\leq C\,|F|^{\frac{1}{2}}
\quad\mbox{and}\quad
\norm{\bpsi_{F,\vz}}_{\vL^2(K)}
\leq C\, |K|^{\frac{1}{2}},
\]
and using the following coefficient weight relation 
\eqref{eq:rec-wt} on each face $F$:
\[
(1 - \k_{F}^{K})\b^{-1/2}_K = \frac{1}{\b^{1/2}_{K} + \b^{1/2}_{K'}} 
\leq \frac{\sqrt{2}}{(\b_{K} + \b_{K'})^{1/2}}=\frac{1}{\b_F^{1/2}},
\]
the local error indicator ${\eta}_{K,0}$ has the following bound:
\[
\begin{aligned}
{\eta}_{K,0}
&= \norm{\b^{-1/2} \btau^*_h - \b^{1/2}\vu_h }_{\vL^2(K)} 
= \norm{\b^{-1/2} (\btau^*_h -\btau_h)}_{\vL^2(K)}
\\[2mm]
\leq &\sum_{F\subset \p K\backslash \p \Om}\, \sum_{\vz\subset \cN_h(F)} 
\frac{1}{2|F|\b_F^{1/2}}
\abs{\int_F\jump{\b \vu_h\cdot \vn}{F}\lam_{\vz}\,dS} 
\norm{\bpsi_{F,\vz}}_{\vL^2(K)}
\\[2mm]
\leq & 
\sum_{F\subset \p K\backslash \p \Om} C \left(\frac{|K|}{|F|}\right)^{1/2}
\norm{\b_F^{-1/2} \jump{\b \vu_h\cdot \vn}{F}}_{L^2(F)}.
\end{aligned}
\]
For any element with a boundary face, thanks to the setting for problem 
\eqref{eq:pb-fem}, that the Dirichlet data can be exactly represented by an $\ND_0$ 
vector field's tangential trace, the degrees of freedom on any boundary face do not 
contribute to the approximation error in that element. This completes the proof of 
the lemma.
\end{proof}

\begin{lemma}[Equivalence of ${\eta}_{K,\perp}$]
\label{lem:eff-p}
Under Assumption \ref{asp:mnt}, there exists a constant $c > 0$ independent of the 
jumps of the coefficients such that for any $K \in \cT_h$
\begin{equation}
\label{eq:eff-p}
c\,{\eta}_{K,\perp} \leq \sum_{e\in \cE_h(K)} 
\sum_{F\subset \om_{e,F}}h_F^{1/2} 
\norm{\mu_F^{1/2} \jump{(\mu^{-1} \curlt \vu_h)\cross \vn}{F}}_{\vL^2(F)}.
\end{equation}
\end{lemma}

\begin{proof}
The proof of this lemma uses the setting in the edge 
patch's illustration of Figure \ref{fig:geom-patche}. The edge patch $\om_e$ 
consists of 4 tetrahedra, and the 
following proof generalizes without essential changes to the case when there are 
more than 4 tetrahedra in $\om_e$.

Without loss of generality, the element of interest $K$ is assumed to be $K_1$ in 
Figure \ref{fig:geom-patche}. First performing the partition of unity for 
$\bsig_{h,K} = \mu^{-1}\curlt\vu_h\at{K}$, which is a constant vector and can be 
represented by $\ND_0(K)$ vector fields:
\begin{equation}
\label{eq:eff-pe}
\begin{aligned}
&{\eta}_{K,\perp}
= \norm{\mu^{\frac{1}{2}} {\bsig}^*_h - \mu^{-\frac{1}{2}} \curlt \vu_h}_{\vL^2(K)} 
= 
\norm{\mu^{\frac{1}{2}} \sum_{e\in \cE_h(K)} ({\bsig}^*_h - \bsig_{h,K})
\cdot \vt_e\,\bvphi_e}_{\vL^2(K)}
\\[2mm]
\leq&  \sum_{e\in \cE_h(K)}
\norm{\mu^{\frac{1}{2}}  ({\bsig}^*_h - \bsig_{h,K})
\cdot \vt_e\,\bvphi_e}_{\vL^2(K)}
\leq  \sum_{e\in \cE_h(K)} \mu^{\frac{1}{2}}_K
\bigl\vert({\bsig}^*_h - \bsig_{h,K})
\cdot \vt_e \bigr\vert\,\norm{\bvphi_e}_{\vL^2(K)}.
\end{aligned}
\end{equation}

By the fact that $\norm{\bvphi_e}_{\vL^2(K)} \leq C |K|^{\frac{1}{2}}$, 
the rest of the proof is to establish the equivalence, for every edge $e$, of 
$\bigl\vert({\bsig}^*_h - \bsig_{h,K})\cdot \vt_e \bigr\vert$ with the 
coefficient-weighted tangential jump term in the residual-based estimator.

For the rest of the proof let us assume the edge of interest is $e$ in Figure 
\ref{fig:geom-patche}. 
Before moving on to different coefficient distribution scenarios in this edge patch, 
first by the local recovery \eqref{eq:rec-local} and the $\ND_0$ basis function 
construction \eqref{eq:basis-nd0}, it is straightforward to check that
\[
{\bsig}^*_h \cdot \vt_e = \frac{1}{|\wh{\om}_{e,F}|}\sum_{F \subset \wh{\om}_{e,F}}
\int_{F} (\bsig_{h,F}\cdot \vt_e)\,dS.
\]

The first case is when $\wh{\om}_e = K = K_1$, then $\wh{\om}_{e,F} = F_1\cup F_4$. 
Using the geometric relation that for any 
$\vv \cdot \vt_e =\vn_F\cross(\vv\cross \vn_F)\cdot \vt_e$ if $\vt_e$ lies on the 
planar surface $F$, and the definition of the weighted average $\bsig_{h,F_i}$ in 
\eqref{eq:rec-favg}, yields
\begin{equation}
\label{eq:eff-pt}
\begin{aligned}
&({\bsig}^*_h - \bsig_{h,K})\cdot \vt_e = 
\frac{1}{|\wh{\om}_{e,F}|}\sum_{i \in \{1,4\}}
\int_{F_i} (\bsig_{h,F_i} - \bsig_{h,K})\cdot \vt_e\,dS
\\
= &\frac{1}{|\wh{\om}_{e,F}|}\sum_{i \in \{1,4\}}
\int_{F_i} \vn_{F_i}\cross\big((\bsig_{h,F_i} - \bsig_{h,K})\cross\vn_{F_i}\big) 
\cdot \vt_e \,dS
\\
=& \frac{1}{|\wh{\om}_{e,F}|}\sum_{i \in \{1,4\}}
\int_{F_i}(1 - \g_{F_i}^K)\jump{(\mu^{-1} \curlt \vu_h)\cross\vn}{F_i}
\cdot (\vt_e\cross\vn_{F_i})\,dS.
\end{aligned}
\end{equation}
By the coefficient weight defined in \eqref{eq:rec-wt}, for $F_1$ we have
\[
\mu_K^{1/2}(1 - \g_{F_1}^K) =  \frac{1}{\mu^{-1/2}_{K} + \mu^{-1/2}_{K_2}} 
 \leq \frac{\sqrt{2}}{(\mu^{-1}_{K} + \mu^{-1}_{K_2})^{1/2}} = \mu_{F_1}^{1/2}.
\]
By Cauchy-Schwarz inequality and the triangle 
inequality
\begin{equation}
\label{eq:eff-pj1}
\mu^{1/2}_K
\bigl\vert({\bsig}^*_h - \bsig_{h,K})
\cdot \vt_e\bigr\vert \leq 
\frac{1}{2|\wh{\om}_{e,F}|}\sum_{i \in \{1,4\}} 
|F_i|^{1/2}  \norm{\mu_{F_i}^{1/2} 
\jump{(\mu^{-1} \curlt \vu_h)\cross \vn}{F_i}}_{\vL^2(F_i)}.
\end{equation}

Then using the shape regularity of the mesh, i.e.
$|F_i|^{1/2} |K|^{1/2} |\wh{\om}_{e,F}|^{-1}\leq C \, h_{F_i}^{1/2}$ 
for any $F_i$ in this edge patch, we have
\begin{equation}
\label{eq:eff-pj2}
\norm{\mu^{\frac{1}{2}}  ({\bsig}^*_h - \bsig_{h,K})
\cdot \vt_e\,\bvphi_e}_{\vL^2(K)}\leq
\sum_{F\subset \wh{\om}_{e,F}} h_F^{1/2}
\norm{\mu_{F}^{1/2} \jump{(\mu^{-1} \curlt \vu_h)\cross \vn}{F}}_{\vL^2(F)}.
\end{equation}
A variant of the first case is that $K=K_1\subsetneq \wh{\om}_{e,F}$. Assume 
$\wh{\om}_{e,F} = K_1\cup K_2$, then $\wh{\om}_{e,F} = F_1\cup F_2 \cup F_4$. By the 
definition of $\wh{\om}_{e,F}$ in \eqref{eq:patch-ef}, 
$\mu_{K_1}^{-1} = \mu_{K_2}^{-1} = \min_{i=1,\dots,4} \mu_{K_i}^{-1}$. The 
proof of the bound \eqref{eq:eff-pj1} for this variant shares almost the 
same argument with above, except there will be one extra term comparing to 
\eqref{eq:eff-pt}, and 
it can be rewritten as follows:
\begin{equation}
\label{eq:eff-pj3}
\begin{aligned}
&\frac{1}{|\wh{\om}_{e,F}|} \int_{F_2} (\bsig_{h,F_2} - \bsig_{h,K})\cdot \vt_e\,dS
\\
= & 
\frac{1}{|\wh{\om}_{e,F}|} \int_{F_2} 
\Bigl[(\bsig_{h,F_2} -\bsig_{h,K_2} )
+ (\bsig_{h,K_2} - \bsig_{h,K})\Bigr]\cdot \vt_e\,dS
\\
= & 
\frac{1}{|\wh{\om}_{e,F}|} \int_{F_2} (1 - \g_{F_2}^{K_2})
\jump{(\mu^{-1} \curlt \vu_h)\cross\vn}{F_2} \cdot (\vt_e\cross\vn_{F_2})\,dS
\\
& + \frac{1}{|\wh{\om}_{e,F}|} \int_{F_2} 
\jump{(\mu^{-1} \curlt \vu_h)\cross\vn}{F_1} \cdot (\vt_e\cross\vn_{F_1})\,dS.
\end{aligned}
\end{equation}

Using the the shape regularity of the edge patch ($c|F_2|\leq |F_1| \leq C |F_2|$), 
and the fact that
\[
\mu_K^{1/2}(1 - \g_{F_2}^{K_2}) = \mu_{K_2}^{1/2}(1 - \g_{F_2}^{K_2})
\leq \mu_{F_2}^{1/2}, \; \text{ and }\; 
\mu_K^{1/2} = \frac{2}{\mu_{K}^{-1/2}+\mu_{K_2}^{-1/2}} \leq 2 \mu_{F_1}^{1/2},
\]
we reach the following estimate
\begin{equation}
\label{eq:eff-pj4}
\begin{aligned}
&\mu^{1/2}_K
\abs{\frac{1}{|\wh{\om}_{e,F}|} \int_{F_2} (\bsig_{h,F_2} - \bsig_{h,K})\cdot 
\vt_e\,dS}
\\
\leq& 
\frac{C}{|\wh{\om}_{e,F}|}\sum_{i \in \{1,2\}} 
|F_i|^{1/2}  \norm{\mu_{F_i}^{1/2} 
\jump{(\mu^{-1} \curlt \vu_h)\cross \vn}{F_i}}_{\vL^2(F_i)}.
\end{aligned}
\end{equation}

Thus the estimate \eqref{eq:eff-pj2} follows. If $\wh{\om}_e$ contains more 
elements, the same argument with above applies, with all the unweighted extra terms 
involve only the interior faces of $\wh{\om}_e$. This completes the proof for the 
first case.

The second case when $K=K_1\not\subset \wh{\om}_e$, yet $K$ is adjacent to 
$\wh{\om}_e$. Assume $K_2 = \wh{\om}_e$, i.e., $\wh{\om}_{e,F} = F_1\cup F_2$. A 
similar split as 
\eqref{eq:eff-pt} applies
\[
({\bsig}^*_h - \bsig_{h,K})\cdot \vt_e = 
\frac{1}{|\wh{\om}_{e,F}|}\sum_{i \in \{1,2\}}
\int_{F_i} (\bsig_{h,F_i} - \bsig_{h,K})\cdot \vt_e\,dS.
\]
The $F_1$ term can be estimated the same with \eqref{eq:eff-pj1}. The $F_2$ term can 
be rewritten as \eqref{eq:eff-pj3}. This time we use 
$ \mu_K^{-1} \geq \mu_{K_2}^{-1} = \min_{i=1,\dots,4}\mu_{K_i}^{-1}$, this implies
\[
\mu_{K}^{1/2} \leq \frac{2}{\mu_{K}^{-1/2} + \mu_{K_2}^{-1/2} }
\leq 2 \mu_{F_1}^{1/2},
\]
thus the estimate \eqref{eq:eff-pj4} follows, which, under some backtracking, 
confirms the validities of estimates \eqref{eq:eff-pj1} and \eqref{eq:eff-pj2}. If 
$\wh{\om}_e$ contains more elements than $K_2$, same argument applies as long as 
$\mu_K^{-1} \geq \mu_{\wh{\om}_e}^{-1}$ and the shape regularity holds for the edge 
patch of interest. This completes the proof for the second case.

The third case is that $K=K_1\not\subset \wh{\om}_e$, nor is $K$ neighboring to 
$\wh{\om}_e$. Assuming $\wh{\om}_e = K_3$, then $\wh{\om}_{e,F} = F_2\cup F_3$. 
The same split with \eqref{eq:eff-pt} applies, but this time on face $F_2$ and $F_3$,
\[
({\bsig}^*_h - \bsig_{h,K})\cdot \vt_e = 
\frac{1}{|\wh{\om}_{e,F}|}\sum_{i \in \{2,3\}}
\int_{F_i} (\bsig_{h,F_i} - \bsig_{h,K})\cdot \vt_e\,dS.
\]
In this lemma, Assumption \ref{asp:mnt} holds. Without loss of generality, we assume 
the monotone path from $K=K_1$ to $K_3$ is through $K_2$.
The $F_2$ term can be estimated exactly like previous case, because $\mu_K^{-1} \geq 
\mu_{K_2}^{-1} \geq \mu_{K_3}^{-1} = \min_{i=1\dots,4} \mu_{K_i}^{-1}$. For the 
$F_3$ term, using the same trick as \eqref{eq:eff-pj3} yields:
\[
\begin{aligned}
&\frac{1}{|\wh{\om}_{e,F}|} \int_{F_3} (\bsig_{h,F_3} - \bsig_{h,K})\cdot \vt_e\,dS
\\
= & 
\frac{1}{|\wh{\om}_{e,F}|} \int_{F_3} 
\Bigl[(\bsig_{h,F_3} -\bsig_{h,K_3}) + (\bsig_{h,K_3} -\bsig_{h,K_2}) + 
(\bsig_{h,K_2}- \bsig_{h,K})\Bigr]
\cdot \vt_e\,dS
\\
= & 
\frac{1}{|\wh{\om}_{e,F}|} \int_{F_3} (1 - \g_{F_3}^{K_3})
\jump{(\mu^{-1} \curlt \vu_h)\cross\vn}{F_3} \cdot (\vt_e\cross\vn_{F_3})\,dS
\\
& + \frac{1}{|\wh{\om}_{e,F}|}\sum_{i\in \{1,2\}} \int_{F_3} 
\jump{(\mu^{-1} \curlt \vu_h)\cross\vn}{F_i} \cdot (\vt_e\cross\vn_{F_i})\,dS.
\end{aligned}
\]
By the quasi-monotonicity of the coefficient on this edge patch again, we have
\[
\mu_K^{1/2}(1 - \g_{F_3}^{K_3}) \leq \mu_{K_3}^{1/2}(1 - \g_{F_3}^{K_3})
\leq \mu_{F_3}^{1/2}, \; \text{ and }\; 
\mu_K^{1/2} \leq \frac{2}{\mu_{K_2}^{-1/2}+\mu_{K_3}^{-1/2}} \leq 2 \mu_{F_2}^{1/2},
\]
therefore, the estimate for the $F_3$ term is similar to \eqref{eq:eff-pj4}, with 
one extra face included due to the fact that the inequality is passed through an 
intermediate element along the monotone path
\begin{equation}
\label{eq:eff-pj5}
\begin{aligned}
&\mu^{1/2}_K
\abs{\frac{1}{|\wh{\om}_{e,F}|} \int_{F_3} 
(\bsig_{h,F_3} - \bsig_{h,K})\cdot \vt_e\,dS}
\\
\leq& 
\frac{C}{|\wh{\om}_{e,F}|}\sum_{i \in \{1,2,3\}} 
|F_i|^{1/2}  \norm{\mu_{F_i}^{1/2} 
\jump{(\mu^{-1} \curlt \vu_h)\cross \vn}{F_i}}_{\vL^2(F_i)}.
\end{aligned}
\end{equation}
Consequently, estimates \eqref{eq:eff-pj1} and \eqref{eq:eff-pj2} follow for the 
third case. If the reader walks through the proof, one will find that more 
tetrahedra being contained in $\wh{\om}_e$ than 1 does not change the essential part 
of the proof because of the existence of the monotone path. This completes the proof 
of the lemma.
\end{proof}

\begin{theorem}[Local Efficiency of ${\eta}_K$]
\label{thm:eff}
Under Assumption \ref{asp:mnt}, there exists a constant $c > 0$ independent of the 
jumps of the coefficients such that for any $K \in \cT_h$:
\begin{equation}
\label{eq:eff}
c\, {\eta}_K \leq \enorm{\vu - \vu_h}_{\om_{K,F}} 
+ \osc(\vf,\mu,\b;\om_{K,F}), 
\end{equation}
where $\osc(\vf,\mu,\b;\om_{K,F})$ is 
the oscillation of the data within $\om_{K,F}$
\[
\osc(\vf,\mu,\b;\om_{K,F}) = \left\{ \sum_{K\subset\om_{K,F}} 
\Bigl( \b_K^{-1}h_K^2\norm{\divv\vf}_{L^2(K)}^2
+ \mu_K h_K^2\norm{\vf-\vf_h}_{\vL^2(K)}^2
\Bigr) \right \} ^{1/2}.
\]
\end{theorem}

\begin{proof}
By the residual-based estimator local efficiency estimate \eqref{eq:eff-residual}, 
Lemma \ref{lem:eff-0} and \ref{lem:eff-p} which show the recovery-based $\eta_{K,0}$ 
and $\eta_{K,\perp}$ can be bounded the face jumps in the residual-based estimator, 
it suffices to show that the local recovery-based residual term is locally 
efficient. Applying the triangle inequality for ${\eta}_{K,R}$ gives:
\[
{\eta}_{K,R} \leq \mu^{1/2}_K h_K \left(
\norm{\vf - \b \vu_h - \curlt (\mu^{-1}\curlt \vu_h)}_{\vL^2(K)} + 
\norm{\curlt(\mu^{-1}\curlt \vu_h - \bsig^*_h)}_{\vL^2(K)}\right),
\]
which, together with a standard inverse inequality 
and~\eqref{eq:eff-residual}, shows that
\[
c\,{\eta}_{K,R} \leq \enorm{\vu - \vu_h}_K + {\eta}_{K,\perp}.
\]
This completes the proof of the theorem.
\end{proof}

\section{Numerical Experiments}
This section reports numerical results of our estimator on several 
three dimensional $\vhcrl$ interface test problems. 

The numerical tests are implemented under $i$FEM (see~\cite{Chen.L2008c}) framework 
in MATLAB. Initial meshes are generated by the MATLAB built-in 
\texttt{DelaunayTri} and \texttt{distmesh} (see \cite{Persson04}). At each 
iteration, let $\cS_h$ be a subset of $\cT_h$ whose elements satisfy
\[
\sum_{K\in \cS_h} \eta_K^2 \geq \theta \sum_{K\in \cT_h} \eta_K^2,
\]
where the $\eta_K $ is evaluated using the recovered quantities computed by weighted 
$\vL^2$-projections through multigrid $V(3,2)$-cycle iterations. 
This procedure is analyzed in \cite{Xu04} for diffusion 
problem, and is proved to be equivalent to the local weighted averaging. 
The marking parameter $\theta$ is chosen to be $0.2$. All elements in $\cS_h$
are refined locally by bisecting the longest edge, and some neighboring elements of 
$\cS_h$ are refined to preserve conformity of the triangulation.

To measure the global reliability of the \emph{a posteriori} error estimator, 
we show comparisons of different measures in the 
each example's table of comparison. $n$ is the number 
of levels of refinement. The $N_n$ the dimension of $\cE_{h,n}$ in the $n$-th 
level triangulation, in our case, it is the number of degrees of freedom. The 
\emph{effectivity index} for each estimator at the $n$-th level is:
 \[
\text{eff-index} := \frac{\eta_n}{\enorm{\vu - \vu_{h,n}}},
\]
where $\eta_n$ is the error estimator, and $\vu_{h,n}$ is the finite element 
approximation at the $n$-th level of triangulation.

The \emph{orders of convergence} are computed for both 
$\eta$ and $\enorm*{\vu - \vu_h}$.  $r_{\eta}$ and $r_{\text{err}}$ are defined 
as the slope for the line of $\eta_n$ and $\enorm*{\vu - \vu_{h,n}}$ in the 
log-log scale plot, such that
\[
\ln \eta_n \sim -r_{\eta} \ln N_n + c_1,\quad\text{and}\quad
\ln \enorm{\vu - \vu_{h,n}} \sim -r_{\text{err}} \ln N_n + c_2.
\]
In the convergence rate plot, the log of degrees of freedom is the horizontal 
axis, and the log of the error/estimator is the vertical axis. The order of 
convergence is optimal when $r_{\eta}$ and $r_{\text{err}}$ are approximately 
$1/3$.

In first two examples with known true solutions, the 
adaptive mesh refinement procedure is terminated when the true 
relative error
\[
\text{rel-error} := \enorm{ \vu - \vu_h}/\enorm{\vu}\leq \texttt{Tol}.
\]

For comparison, numerical results involve some of the following error 
estimators other than the recovery estimator in~\eqref{eq:est-locavg}:

\begin{itemize}[itemsep=5pt,topsep=0pt,leftmargin=20pt]

\item[1.] The residual estimator in \cite{Beck-Hiptmair-Hoppe-Wohlmuth}:
\begin{equation} 
\label{eq:est-residual}
\begin{aligned}
\eta_{K,Res}^2 
&= \mu_K h_K^2\norm{\vf - \b \vu_h - \curlt(\mu^{-1} \curlt 
\vu_h)}_{\vL^2(K)}^2 + 
\beta_K^{-1} h_K^2\norm{\divv(\b \vu_h- \vf)}_{\vL^2(K)}^2
\\[2mm]
&  + \sum_{F\in \cF_h(K)} \frac{h_F}{2} \left( \b_F^{-1} 
\norm{\jump{\b\vu_h\cdot \vn_F}{F}}_{L^2(F)}^2 + 
\mu_F\norm{\jump{(\mu^{-1} \curlt \vu_h)\cross \vn}{F}}_{\vL^2(F)}^2 \right),
\end{aligned}
\end{equation}
and $\eta_{Res}^2 =  \sum_{K\in \cT_h} \eta_{K,Res}^2$, where $\mu_F^{-1}$ and 
$\b_F$ are the arithmetic averages of $\mu^{-1}$ and $\b$, 
respectively, on elements sharing the face $F$. Note that this estimator is 
weighted appropriately and may be viewed as the extension of the residual 
estimator in \cite{Bernardi-Verfurth, Petzoldt02} for the diffusion interface 
problem to the $\vhcrl$ interface problem.

\item[2.] The Zienkiewicz-Zhu (ZZ) based error estimator in \cite{Nicaise05} 
using the coefficient-weighted norm~\eqref{eq:norm-energy}:
\begin{equation}
\label{eq:est-zz}
\begin{aligned}
\eta_{K,ZZ}^2 
=& \norm{\mu^{-1/2}\cR_{\perp}(\curlt \vu_h) 
- \mu^{-1/2}\curlt \vu_h}_{\vL^2(K)}^2 
\\
&+ \norm{\b^{1/2}\cR_{0}(\vu_h) - \b^{1/2} \vu_h}_{\vL^2(K)}^2,
\end{aligned}
\end{equation}
and $\eta_{ZZ}^2  = \sum_{K\in \cT_h} \eta_{K,ZZ}^2$. Both recovered quantities 
$\cR_{\perp}( \curlt \vu_h)$ and $\cR_{0}( \vu_h)$ are in the continuous 
piecewise linear vector fields space $\PC:= 
\{\vp\in \vHs{1}{\Omega}: \vp\at{K}\in \vP_1(K) \}$, 
and their nodal values at any vertex 
$\vz\in \cN_h$ are:
\[
\cR_{\perp}( \curlt \vu_h)\at{\vz} = \frac{1}{|\om_{\vz}|} \int_{\om_{\vz}} 
 \curlt \vu_h\,d\vx\;\text{and}\; 
\cR_{0}( \vu_h)\at{\vz} = \frac{1}{|\om_{\vz}|} \int_{\om_{\vz}}  \vu_h\,d\vx.
\]

\item[3.] The Zienkiewicz-Zhu (ZZ) flux based error estimator in 
\cite{Nicaise05} with weight suited to the coefficient-weighted 
norm~\eqref{eq:norm-energy}:
\begin{equation}
\label{eq:est-zzflux}
\begin{aligned}
\eta_{K,ZZ,f}^2 
=& \norm{\mu^{1/2}\cR_{\perp}(\mu^{-1}\curlt \vu_h) - \mu^{-1/2}\curlt 
\vu_h}_{\vL^2(K)}^2
\\
&+ \norm{\b^{-1/2}\cR_{0}(\beta  \vu_h) - \b^{1/2} \vu_h}_{\vL^2(K)}^2
\end{aligned}
\end{equation}
and $\eta_{ZZ,f}^2 = \sum_{K\in \cT_h} \eta_{K,ZZ,f}^2$. Both recovered 
quantities $\cR_{\perp}(\mu^{-1}\curlt \vu_h)$ and $\cR_{0}(\b  \vu_h)$ are in 
$\PC$ as well, and their nodal values at any vertex $\vz\in \cN_h$ are:
\[
\cR_{\perp}(\mu^{-1}\curlt \vu_h)\at{\vz} 
= \frac{1}{|\om_{\vz}|} \int_{\om_{\vz}} 
\mu^{-1}\curlt \vu_h\,d\vx\quad\text{and}\quad 
\cR_{0}(\b \vu_h)\at{\vz} 
= \frac{1}{|\om_{\vz}|} \int_{\om_{\vz}} \b  \vu_h\,d\vx.
\]


\end{itemize}

\textbf{Example 1}: This example is adapted from a benchmark test problem 
(see~\cite{Cai09,Cai10,Kellogg74})
for elliptic interface problems. The 
computational domain is a narrow slit along $z$-direction: 
$\Om =(-1, 1)^2\cross (-\d,\d) $ with $\d = 0.2$. The true 
solution $\vu$ is given in cylindrical coordinates $(r,\theta,z)$:
\[
\vu =\nab \psi =  \nab \bigl(r^{\a} \phi(\theta)\bigr),
\]
where $\phi(\theta) $ takes different values within four different subdomains 
while being glued together using continuity conditions that is firstly invented in 
\cite{Kellogg74}.
The $\mu = 1$, and the $\b$ is given by
\[
\b= 
\begin{cases} 
R & \text{ in } (0, 1)^2\cross(-\d,\d)  \cup (−1, 0)^2\cross (-\d,\d) ,
\\
1 & \text{ in } \Om\backslash \Big((0, 1)^2\cross (-\d,\d)  \cup (−1, 0)^2\cross 
(-\d,\d) \Big).
\end{cases}
\]
Here we set parameters $\a$, $R$ to be
\[ \a = 0.5, \quad R \approx 5.8284271247461907.\]

%

In this example, the tolerance is set to be $ \texttt{Tol} = 0.1$. The numerical 
results of example 1 are in Table~\ref{table:ex1}. It 
shows that to achieve approximately the same level of relative error, the 
number of degrees of freedom needed in the mesh refined by the local indicator 
$\eta_{K,ZZ}$ or $\eta_{K,ZZ,f}$ requires more than twice than the other two. 

\begin{figure}[h]
\centering
\subfloat[Refined Mesh based on $\eta_{ZZ,K}$.]
{
\includegraphics[scale=0.3]{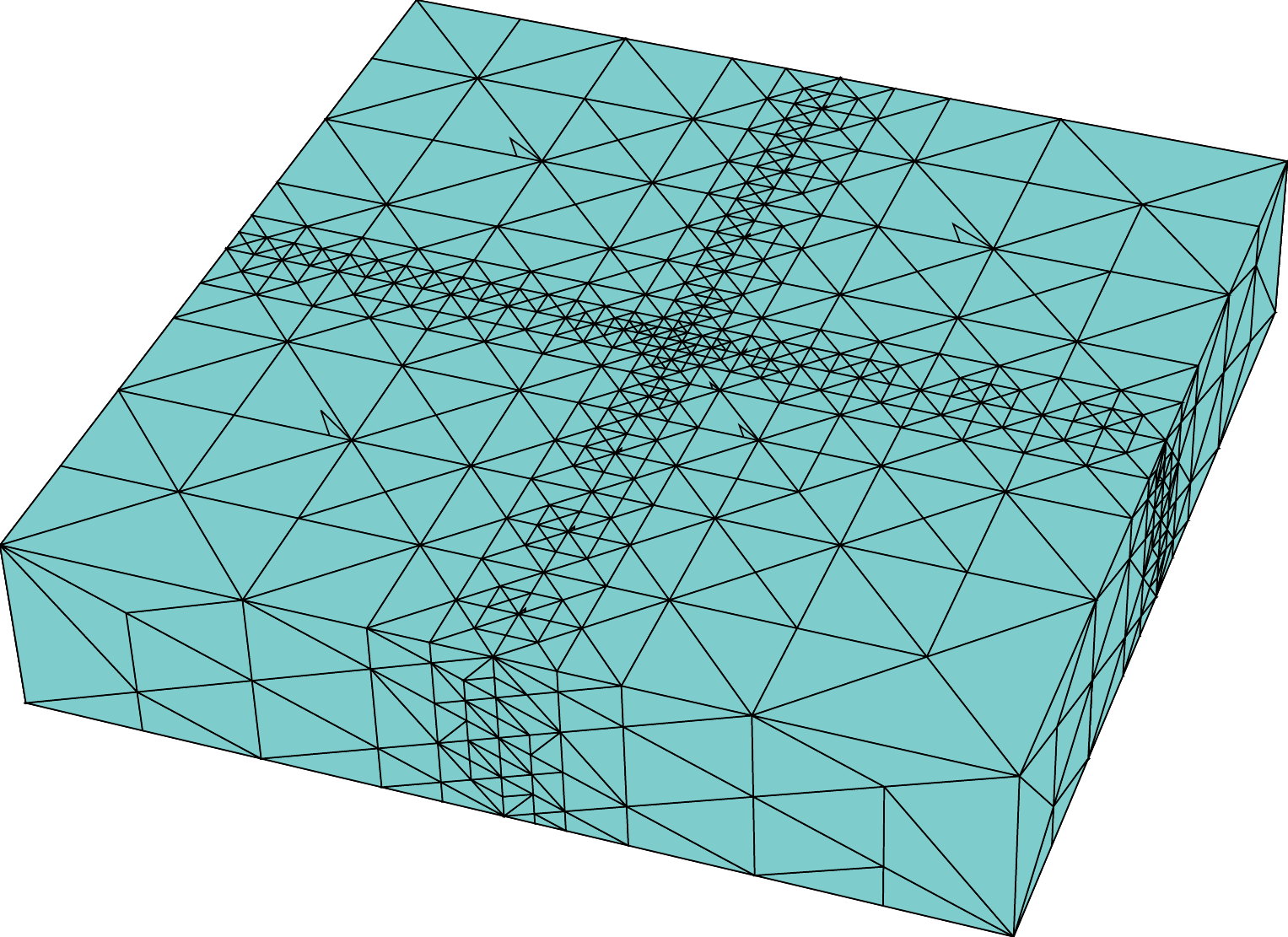}
}
\qquad
\subfloat[Refined Mesh based on $\eta_{ZZ,f,K}$.]
{
\includegraphics[scale=0.3]{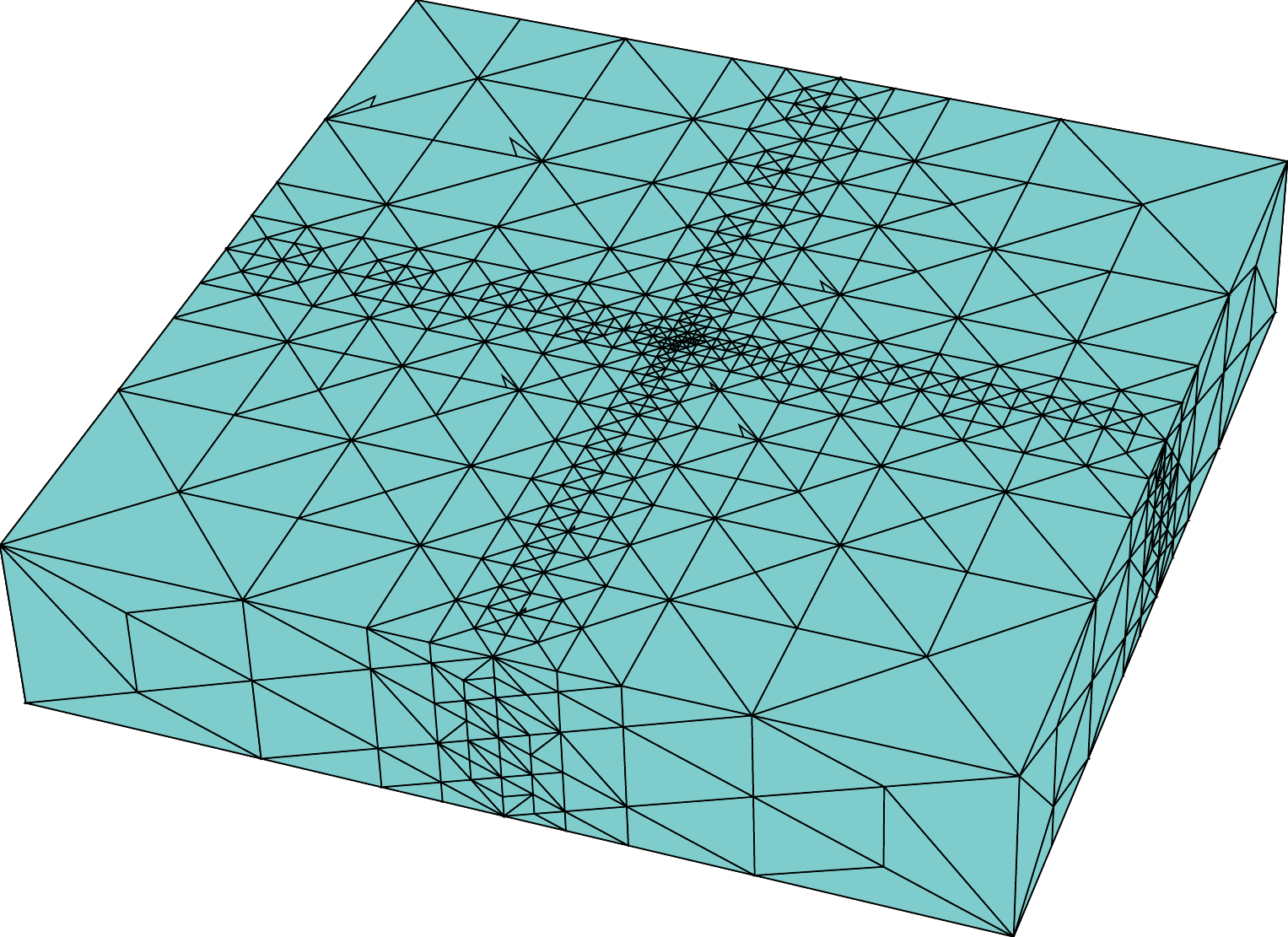}
}
\qquad
\subfloat[Refined Mesh based on $\eta_{K}$.]
{
\includegraphics[scale=0.3]{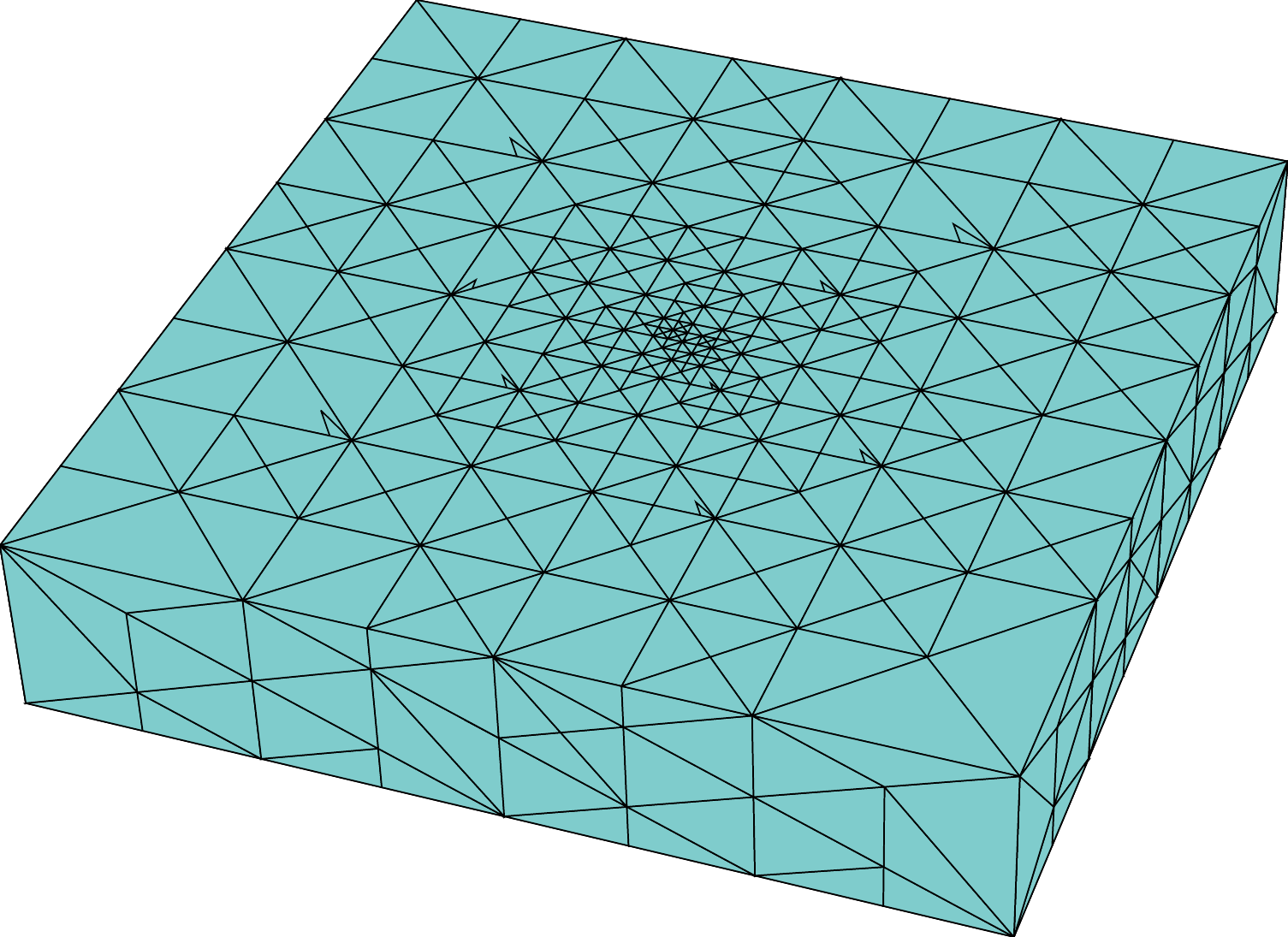}
}
\qquad
\subfloat[Refined Mesh based on $\eta_{Res,K}$.]
{
\includegraphics[scale=0.3]{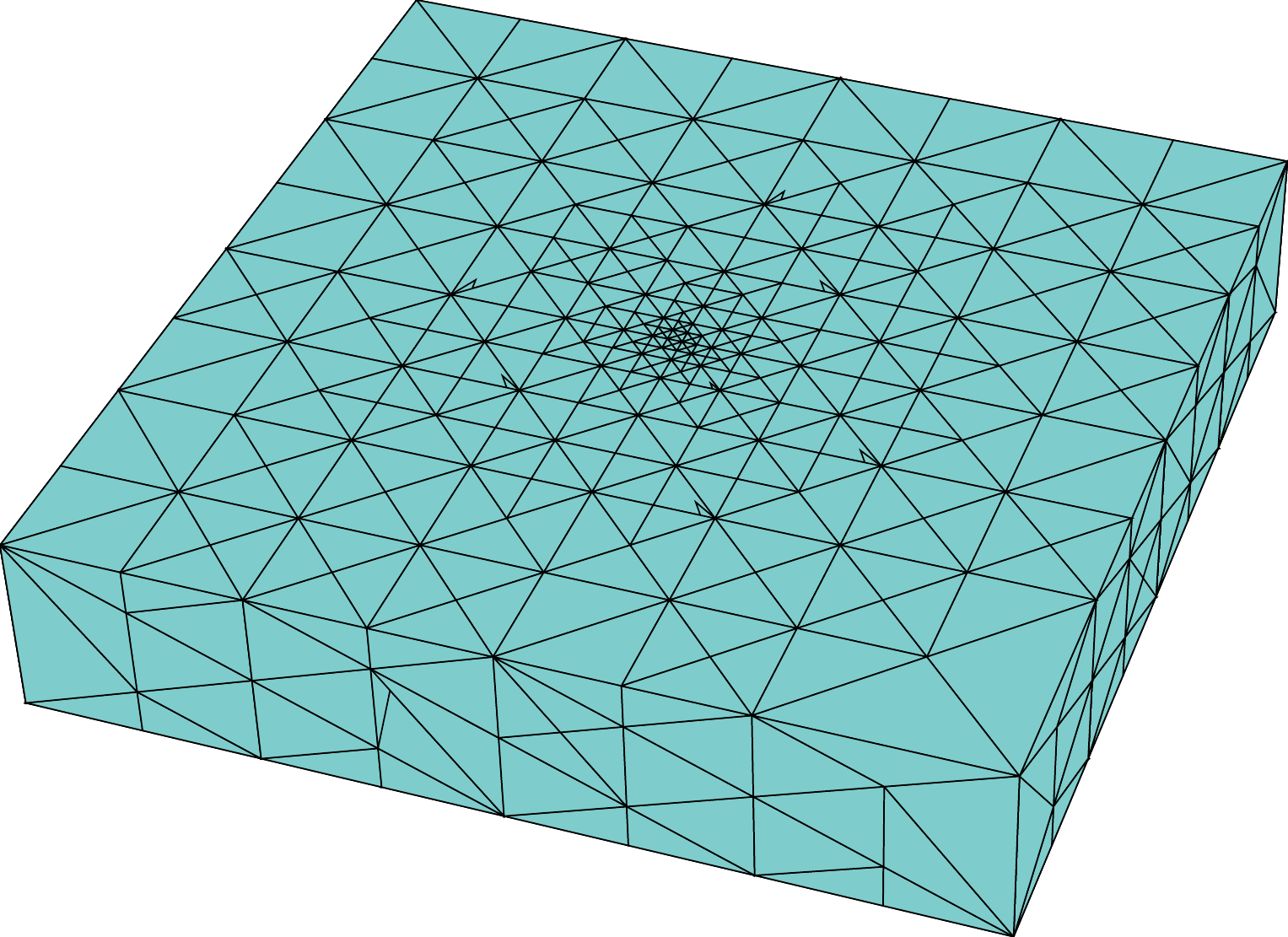}
}
\caption{Mesh result of Example 1}
\label{fig:ex1-mesh}
\end{figure}

The adaptively refined 
mesh generated by each estimator can be found in Figure~\ref{fig:ex1-mesh}. The 
tendency of $\eta_{ZZ}$ or $\eta_{ZZ,f}$ to over-refine those four interfaces is due 
to the fact that recovered quantities enforce unnecessary extra 
continuity conditions of the true quantities. For example, 
$\cR_{\perp}(\mu^{-1}\curlt 
\vu_h)$ and $\cR_{0}(\b \vu_h)$ in~\eqref{eq:est-zzflux} are in 
$\vHs{1}{\Om}$, 
yet for the true solution $\vu$, $\mu^{-1}\curlt \vu \in \vHcrl$ and $\b \vu\in 
\vHdiv$ 
in~\eqref{eq:jump}. 

Overall, the recovery-based error estimator and residual-based 
error estimator lead to the correctly refined mesh, and the recovery-based one 
performs more convincingly showing a less oscillatory convergence, achieving the 
same level of relative error in fewer iterations. More importantly, it exhibits a 
better effectivity index.

\begin{table}[h!]
\caption{Comparison of the estimators in Example 1} 
\centering  
\begin{tabular}{|c|c|c|c|c|c|c|} 
\hline
 & $n$ & \# DoF & rel-error & eff-index & $r_{\eta}$ & $r_{\text{err}}$ \\
\hline
$\eta_{ZZ}$  & $30$ & $39391$ & $0.0898$ & $1.309$ & $0.128$ & $0.235$ \\
\hline
$\eta_{ZZ,f}$  & $26$ & $44111$ & $0.0837$ & $1.495$ & $0.243$ & $0.253$ \\
\hline
$\eta_{Res}$ & $24$ & $18832$ & $0.0873$ & $1.749$ & $0.257$ & $0.301$ \\
\hline
$\eta$       & $18$ & $18649$ & $0.0886$ & $0.820$ & $0.299$ & $0.303$ \\
\hline
\end{tabular}
\label{table:ex1}
\end{table}


\textbf{Example 2}: This example is in the numerical experiments section of 
\cite{Hiptmair-Li-Zou}. The domain is  $\Omega = B_2 = \{(x,y,z): x^2+y^2+z^2 < 
2\}$, 
and the coefficients are given by
\[
\begin{cases} 
\mu = \mu_1 = 1, \, \b = 1 & \text{ in } B_1 = \{(x,y,z): x^2+y^2+z^2 < 1\},
\\
\mu = \mu_2 =  10^6, \, \b = 1  & \text{ in } \Om\backslash B_1.
\end{cases}
\] 
The true solution $\vu $ is given by $\mu \vu_1$ in  $B_1$, and $\mu \vu_2 $ in 
$\Om\backslash B_1$. For the explicit expression please refer 
to~\cite{Hiptmair-Li-Zou}. The $\texttt{Tol} = 0.2$ in this 
example. 

 
In this example, the element residual term 
$\eta_{R}$ in \eqref{eq:est-locavg} is not a higher order term (see 
Figure~\ref{fig:ex2-conv}). The red dashed line is a reference line of a constant 
multiple of $(\#\mathrm{DoF})^{-1/3}$.  
The numerical results of example 2 are in Table~\ref{table:ex2}.
The adaptively refined mesh of each estimator can be found in 
Figure~\ref{fig:ex2-mesh}. 

The refined meshes based on $\eta_{K,Res}$, and $\eta_{K}$ respectively are visually 
similar, the $\eta_{K,ZZ}$ and $\eta_{K,ZZ,f}$ tend to over-refine the region where 
the local coefficient-weighted error is not significant yet $\mu^{-1}\curlt \vu$ is 
discontinuous across the interface. 

\begin{table}[h!]
\caption{Comparison of the estimators in Example 2} 
\centering  
\begin{tabular}{|c|c|c|c|c|c|c|} 
\hline
&$n$ &\# DoF  & rel-error& eff-index & $r_{\eta}$ & $r_{\text{err}}$ \\
\hline
$\eta_{ZZ}$  &$18$& $200692$ & $0.199$ & $4.077$ & Not converging & $0.118$ \\
\hline
$\eta_{ZZ,f}$  &$20$& $99794$ & $0.193$ & $2.527$ & $0.839$ & $0.169$ \\
\hline
$\eta_{Res}$ &$9$ & $63405$ & $0.186$ & $1.761$ & $0.273$ & $0.236$ \\
\hline
$\eta$       &$8$ & $52287$ & $0.193$ & $1.079$ & $0.282$ & $0.251$ \\
\hline
\end{tabular}
\label{table:ex2}
\end{table}

\begin{figure}[h]
\centering
\begingroup
\captionsetup[subfigure]{width=0.48\textwidth}
\subfloat[Relative error distribution of refined mesh based on $\eta_{K,ZZ,f}$ cut 
on $y=0$, it can be observed that the local errors are not evenly distributed.]
{
\includegraphics[scale=0.3]{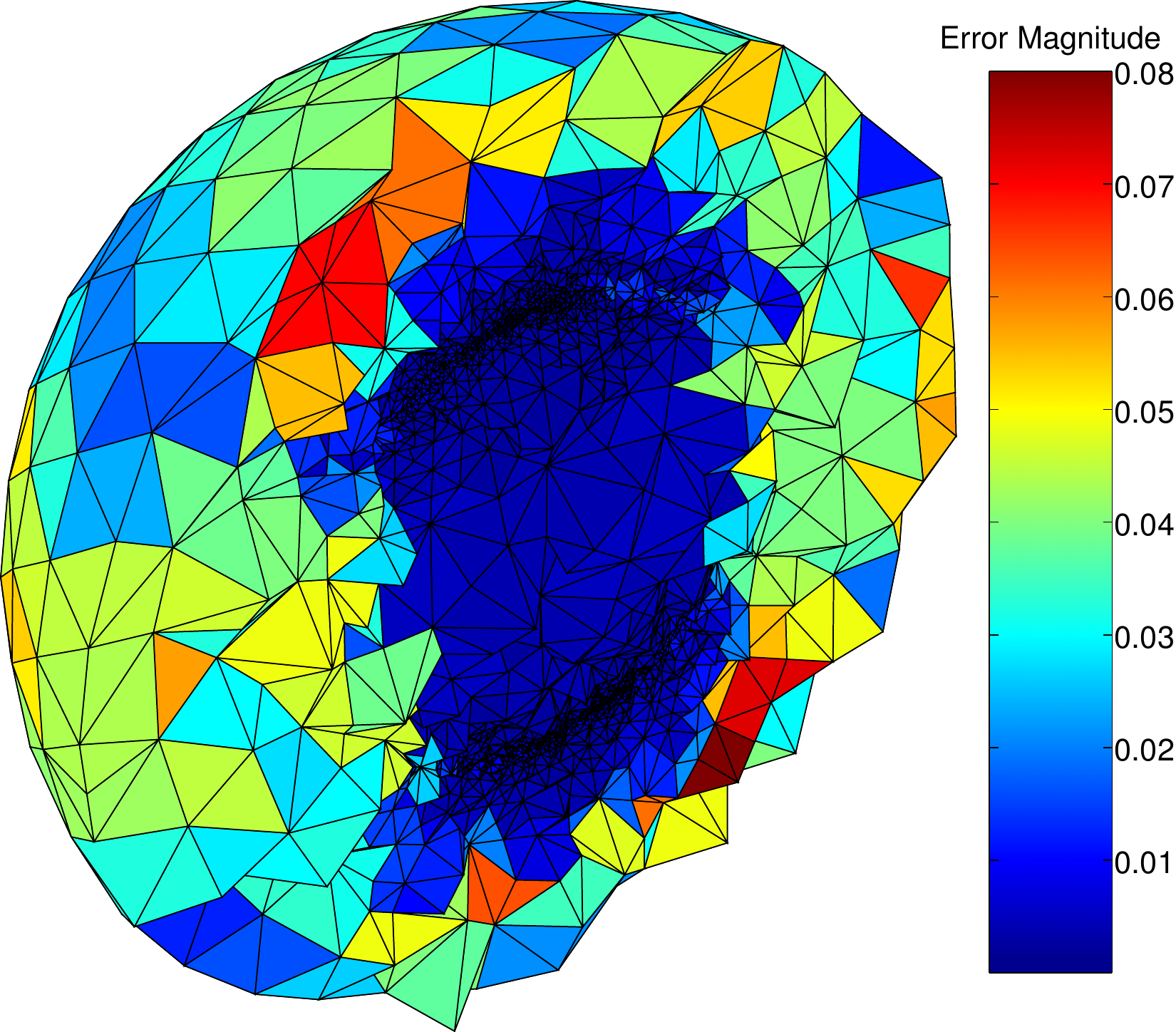}
}
\qquad\qquad
\subfloat[Relative error distribution of refined mesh based on $\eta_{K}$ cut on 
$y=0$, the local errors are more evenly distributed than the mesh refined based on 
$\eta_{K,ZZ,f}$.]
{
\includegraphics[scale=0.3]{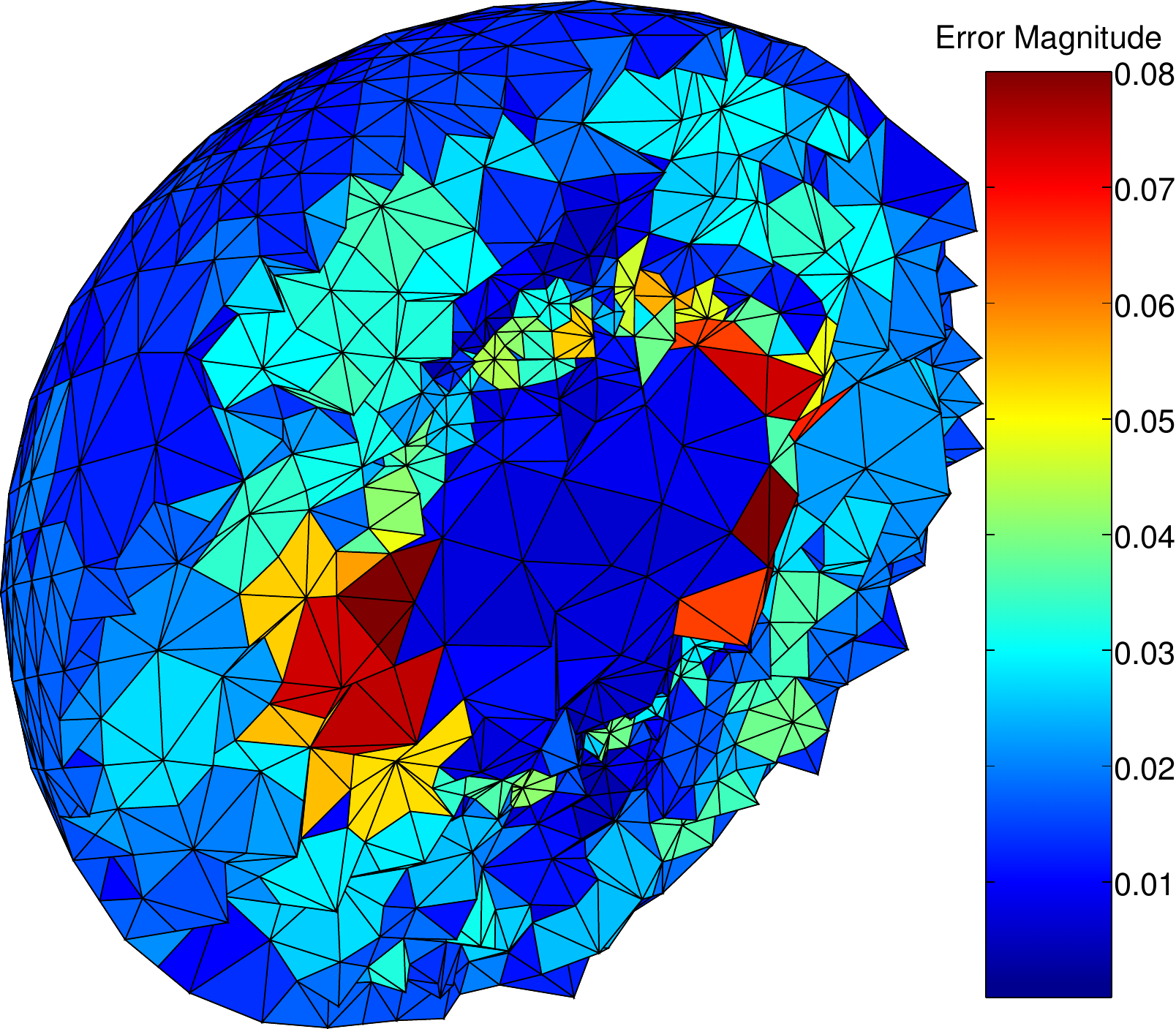}
}
\endgroup
\caption{Relative error distributions result of Example 2}
\label{fig:ex2-mesh}
\end{figure}

\begin{figure}[h]
\centering
\begingroup
\captionsetup[subfigure]{width=0.47\textwidth}
\subfloat[Example 2: convergence of $\eta$, comparing to the element residual term 
$\eta_R$.]
{
\includegraphics[width=0.35\textwidth,height=0.23\textheight]{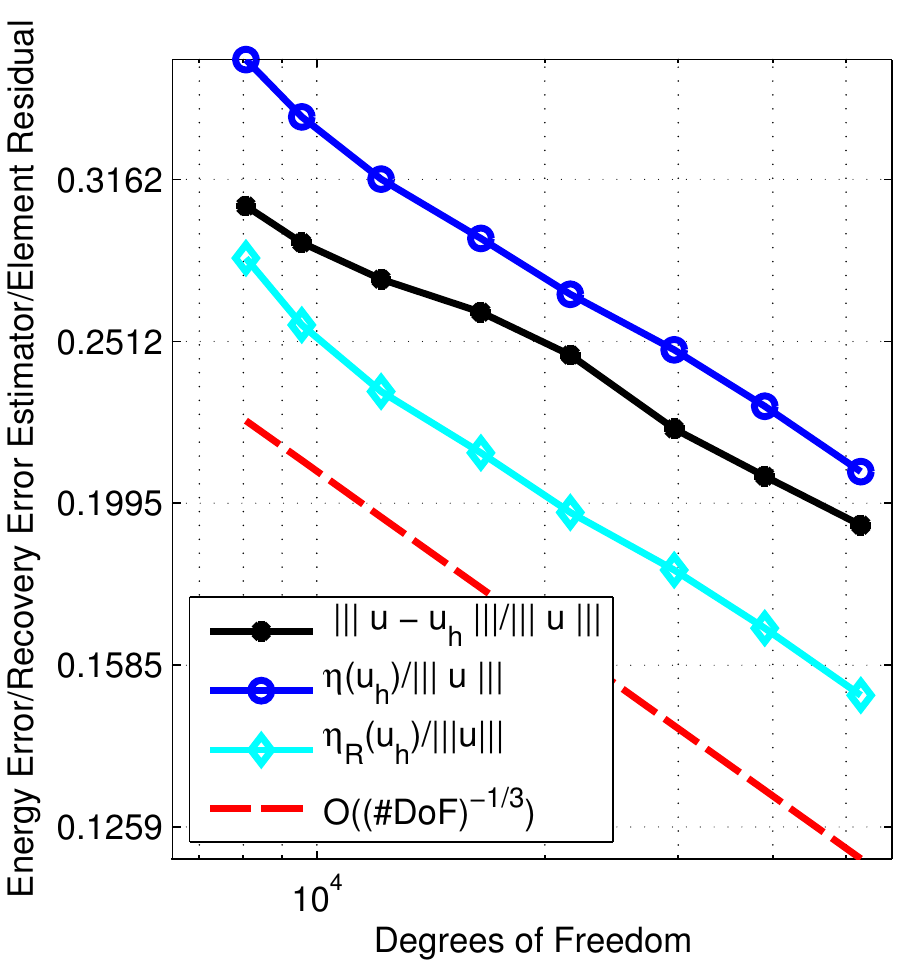}
\label{fig:ex2-conv}
}
\qquad\qquad
\subfloat[Example 3: convergence of $\eta$ versus the recovery-based only 
$\eta_{Rec}= (\eta_{\perp}^2 + \eta_{0}^2)^{1/2}$.]
{
\includegraphics[width=0.35\textwidth,height=0.23\textheight]{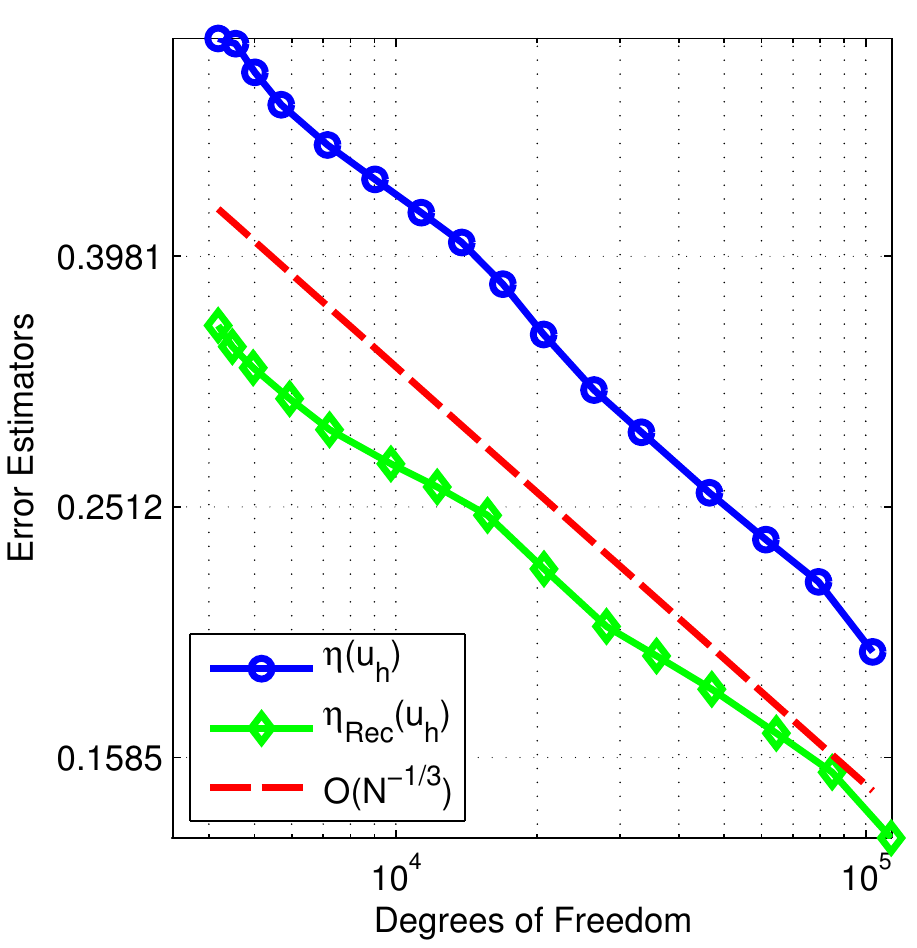}
\label{fig:ex3-conv}
}
\endgroup
\caption{Convergence Results of Example 2 and Example 3}
\end{figure}

\textbf{Example 3}: This example is a widely-used test problem examining the 
performance of adaptive mesh refinement procedure for Maxwell 
equations (e.g. see~\cite{Hiptmair99}). The 
true solution is unknown and not smooth. The homogeneous Dirichlet boundary 
condition is enforced, together with a constant source current 
$\vf = (1,1,1)$. The coefficients are given by:
\[
\begin{cases} 
\mu = 1, \, \b = 1 & \text{ in } \Om_c,
\\
\mu = 1, \, \b = 100  & \text{ in } \Om\backslash \Om_c
\end{cases}
\]
where $\Om_c$ is $\{(x,y,z): |x|,|y|,|z|\leq \frac{1}{2} \}$. In this example, we 
cannot compute the true error, hence we set the stop criterion to be $\eta\leq 
\texttt{Tol}$ with $\texttt{Tol} = 0.16$ at the $n$-th level of triangulation. 

\begin{table}[h!]
\caption{Comparison of the estimators in Example 3} 
\centering  
\begin{tabular}{|c|c|c|c|c|c|} 
\hline
&$n$ &\#(DoF) & $\text{Estimator}$ &$r_{\eta}$ \\
\hline
$\eta_{Res}$ & $21$ & $219993$ & $0.156$ & $0.328$ \\
\hline
$\eta$       & $14$ & $61302$ & $0.152$ & $0.346$ \\
\hline
$\eta_{Rec}$ & $8$ & $15672$ & $0.159$ & $0.230$ \\
\hline
\end{tabular}
\label{table:ex3}
\end{table}

This example illustrates two important aspects: (1) the element residual is 
indispensable in the error estimator in the pre-asymptotic region; 
(2) the iterative refining procedure 
using the residual-based estimator engages much more degrees of freedom than 
the one using the recovery-based estimator, when same stop criterion is used for 
both.

The recovery-based error estimator shows an optimal order of convergence, which 
is $\eta \sim (\#\mathrm{DoF})^{-1/3}$, i.e., $r_{\eta}\approx 1/3$, if the 
local error indicator includes the element residual 
$\eta_K := (\eta_{K,\perp}^2 + \eta_{K,0}^2 + \eta_{K,R}^2)^{1/2}$. 
If the element residual is 
discarded, i.e., the pure recovery-based estimator $\eta_{K,Rec}:= 
(\eta_{K,\perp}^2 + \eta_{K,0}^2)^{1/2}$ is used as the 
local error indicator, the order of convergence for $\eta_{Rec}:= 
\Big(\sum\limits_{K\in \cT_h} \eta_{K,Rec}^2\Big)^{1/2}$ is not optimal (see 
Figure~\ref{fig:ex3-conv}, and Table~\ref{table:ex3}).

From the first two examples, we learn that the effectivity index of the 
recovery-based estimator is in general two times as effective as that of the 
residual-based estimator. For problem with an unknown solution which is quite 
common originated from some real world applications, when setting the stopping 
criterion using the global error estimator, the number of degrees of freedom 
using the residual-based error estimator is 
$(\text{eff-index}_{\text{Res}}/\text{eff-index}_{\text{Rec}})^3$ 
as much as that using the recovery-based error estimator (see 
Table~\ref{table:ex3}).

\appendix
\section{Weighted Helmholtz Decomposition} 
\label{appendix:decomp}

Here we establish a weighted Helmholtz decomposition in light 
of~\cite{Costabel-Dauge-eigenvalue,Costabel-Dauge-Nicaise} tailored for the 
$\vhcrl$ interface problem. The following assumption is needed to guarantee 
that such a decomposition exists with the constant in estimate 
\eqref{eq:rel} is independent of the jumps of the coefficients. In other words, the 
constant in the estimate depends on the jump size of the product of two 
coefficients, and the geometries of the interfaces as well.

\begin{assumption}
\label{asp:coeff}
{\em (i)} The domain $\Om$ is assumed to be convex, simply-connected, and that no 
three or more subdomains share one edge from the triangulation of $\Om$. {\em (ii)} 
The coefficients $\mu$ and $\b$ are assumed to satisfy: 
$C_{\text{min}}\leq \mu_j \b_j \leq C_{\text{max}}$, where 
$C_{\text{min}}$ and $C_{\text{max}}$ are two constants 
independent of the jumps of $\mu_j$ and $\b_j$, or $\Om_j$ on each subdomain 
$\Om_j$.  
\end{assumption}

Firstly, we define some additional function spaces, along with the 
$\vX(\Om,\b)$ in \eqref{eq:sps-decomp}, relevant to the weighted 
Helmholtz decomposition as follows: for any piecewise constant $\a = \a_j$ in 
$\Om_j$: 
\begin{equation}
\label{eq:sps-additional}
\begin{aligned}
&\vHcrlzz = \{\vu\in \vHcrlz: \curlt \vu = 0 \text{ in } \Om\},
\\[1mm]
&\vPCinf = \{\vv\in \vLt: \, 
\vv\at{\Om_j} \in \vC^{\infty}(\Om_j), \, j = 1,\ldots,m\},
\\[1mm]
&\text{and }\PHs{s} = 
\{v \in \Lt: v\at{\Om_j} \in H^s(\Om_j), \, j = 1,\ldots,m \}.
\end{aligned}
\end{equation}

It is well known (see \cite{Girault-Raviart}) that the kernel of 
curl operator, $\vHcrlzz$, is characterized by the gradient field in a 
simply-connected domain:
\begin{lemma}
\label{lem:curlfreepotential}
If $ \Om$ is simply-connected, for any $\vu \in \vHcrlzz $, there exists 
a unique function $ \psi\in \Hoz $ such that $ \vu = \nab \psi $.
\end{lemma}

Since $\cA(\vv,\nab \psi) = (\b \vv,\nab \psi)$, the 
orthogonal complement of $\vHcrlzz = \nab \Hoz $ with respect to 
$\cA(\cdot,\cdot)$ is 
\[
\{\vv\in \vHcrlz: (\b \vv,\nab 
\psi)=0\,\,\forall \psi\in \Hoz\}\subset \vX(\Om,\b). 
\]
To construct a weighted Helmholtz decomposition tailored for the interface 
problem, an analysis of the structure of $\vX(\Om,\b)$ is necessary. Before 
tackling this, the following lemma from~\cite{Costabel-Dauge-Nicaise} is 
needed:
\begin{lemma}
\label{lem:density}
$\vX(\Om,\a) \cap \vPCinf$ is dense in $\vX(\Om,\a)\cap \vPH{1}$ in the 
following norm:
\[
\norm{\vv}_{\vX(\Om)}^2 = \norm{\vv}_{\vLt}^2 + \norm{\curlt\vv}_{\vLt}^2
 + \norm{\divv(\a\vv)}_{\Lt}^2.
\]
\end{lemma}
Now we move on to prove the norm equivalence for certain 
piecewise $\vH^1$-vector fields using the density argument of 
Lemma~\ref{lem:density}.
\begin{lemma}[Norm equivalence for piecewise smooth vector fields]
\label{lem:normeqiv}
For all $\vv \in \vX(\Om,\a)\cap \vPH{1}$, the following identity holds:
\begin{equation}
\label{eq:curldiv}
\sum^m_{j=1}\int_{\Om_j} \a|\nab \vv|^2 = 
\int_{\Om} \left(\a |\curlt \vv|^2 + \a^{-1}|\divv (\a \vv)|^2 \right).
\end{equation}
\end{lemma}
\begin{proof}
By Lemma~\ref{lem:density}, it suffices to establish 
identity~\eqref{eq:curldiv} for 
any $\bphi\in \vX(\Om,\a) \cap \vPCinf$. To this end, using a local identity 
$-\D \bphi = \curlt(\curlt \bphi) - \nab(\divv \bphi)$ and integrating by parts 
on 
each subdomain $\Om_j$ twice give:
\[
\begin{aligned}
&\sum_{j=1}^m \int_{\Om_j} \a|\nab \bphi|^2
= \sum_{j=1}^m\int_{\Om_j} \a_j\left(|\curlt \bphi|^2 + 
|\divv \bphi |^2 \right) + B
\\
\text{with }\; &
B = \sum_{j=1}^m\int_{\p\Om_j} \a_j\Big( (\vn \cdot \nab) \bphi 
+ \vn \cross (\curlt \bphi) - (\divv \bphi)\vn \Big) \cdot \bphi\, dS .
\end{aligned}
\]
Now, it remains to prove that $B = 0$. On any polygonal face with normal vector 
$\vn$ that is represented by the cartesian coordinates in the three dimensional 
space, rather than the local planar coordinates,
$\bphi$ may be decomposed into the normal and tangential components as follows:
\begin{equation}
\label{eq:dcp-vector}
\bphi = (\bphi\cdot \vn)\vn + \bphit\quad \text{ with }\; \bphit = 
\vn\cross(\bphi\cross \vn),
\end{equation}
which, in turn, implies
\begin{equation}
\label{eq:dcp-surfacevector}
\begin{aligned}
&(\vn\cdot \nab)\bphi = \big((\vn\cdot \nab)(\bphi\cdot \vn)\big) \vn
+ (\vn\cdot \nab)\bphit
\\[1mm]
\text{and }&\divv \bphi = \divv \big((\bphi\cdot \vn)\vn\big) + \divv \bphit
= (\vn\cdot \nab)(\bphi\cdot \vn) + \divv \bphit.
\end{aligned}
\end{equation}
By using the following identity (e.g. see \cite{Balanis}) 
\[
\nab(\va\cdot \vb ) = (\va\cdot \nab)\vb + 
(\vb\cdot \nab)\va + \va\cross(\curlt\vb) + \vb\cross(\curlt\va)
\]
and noticing that $\vn$ is a constant vector on a face, we have
\[
\nab (\vn\cdot \bphi) = (\vn\cdot \nab)\bphi  + \vn \cross (\curlt \bphi),
\]
which yields the following by being projected onto each polygonal face
\begin{equation}
\label{eq:dcp-surfacegrad}
\nab_{\top} (\vn\cdot \bphi)= (\vn\cdot \nab)\bphit + \vn \cross (\curlt \bphi),
\end{equation}
where $\nab_{\top}$ is defined as $\nab_{\top} u = (\nab 
u)_{\top}$. It follows from~\eqref{eq:dcp-surfacevector}, 
~\eqref{eq:dcp-surfacegrad}, ~\eqref{eq:dcp-vector}, homogeneous boundary 
condition, and identity~\eqref{eq:dcp-jump} that:
\[
\begin{aligned}
B &= \sum_{j=1}^m\int_{\p\Om_j} \a_j\Big( \nab_{\top} (\vn\cdot \bphi) 
- (\divv \bphit)\vn \Big) \cdot \bphi \,dS
\\[1mm]
&= \sum_{j=1}^m\int_{\p\Om_j} \a_j\Big( \nab_{\top} (\vn\cdot \bphi) \cdot 
\bphit
- (\divv \bphit)(\bphi\cdot \vn) \Big)  \,dS
\\[1mm]
&= \sum_{F\subset \fI} \int_F \Big( \jump{\nab_{\top} 
(\a\bphi\cdot\vn)\cdot \bphit }{F}- 
\jump{(\divv \bphit)(\a\bphi\cdot \vn)}{F}\Big)\,dS
\\[1mm]
&= \sum_{F\subset \fI} \int_F\Big( \nab_{\top}(\jump{\a\bphi\cdot 
\vn}{F})\cdot \bphit 
+\divv\jump{\bphit}{F}(\a\bphi\cdot\vn)
\\[1mm]
&\hspace{0.6in} + \nab_{\top}(\a\bphi\cdot 
\vn)\cdot \jump{\bphit}{F} +
\divv \bphit \jump{\a\bphi\cdot\vn}{F} \Big)\,dS .
\end{aligned}
\]
Now $B=0$ is a direct consequence of the continuity conditions for $\bphi \in 
\vX(\Om,\a) \cap \vPCinf$:
\[
\jump{\bphi\cross \vn}{F} = \vect{0} \text{ and } \jump{\a\bphi\cdot\vn}{F} = 
0 
\quad 
\forall F\subset \fI.
\]
This completes the proof of the lemma.
\end{proof}

\begin{remark}
Lemma \ref{lem:normeqiv} is an extension to the Lemma 3.8 in 
\cite{Girault-Raviart} for Lipschitz polyhedron in the case when only 
homogeneous tangential boundary condition is satisfied for the vector field. It 
uses a similar argument to that of the Theorem 2.3 in 
\cite{Costabel-Dauge-eigenvalue}. In \cite{Costabel-Dauge-eigenvalue}, no 
piecewise constant coefficients are involved, but the technique used shed light 
upon this kind of identity. The result in Lemma \ref{lem:normeqiv} bears the 
same form with an identity valid for $\vPH{2}$ regular vector fields used in 
Lemma 2.2 in \cite{Costabel-Dauge-Nicaise}. In the proof of Lemma 
\ref{lem:normeqiv}, we further exploit the density result in 
\cite{Costabel-Dauge-Nicaise}, which implies this identity in 
\cite{Costabel-Dauge-Nicaise} Lemma 2.2 holds for $\vPH{1}$ regular 
vector fields when the jump conditions are met on the interfaces.
\end{remark}

\begin{theorem}[Weighted Helmholtz decomposition]
\label{thm:dcp-general}
Under Assumption~{\em \ref{asp:coeff}}, for any $\vv\in \vHcrlz$, there 
exist $\psi \in \Hoz$ and $\vw \in \vPH{1}\cap \vX(\Om,\b)$ such that the
decomposition \eqref{eq:dcp} holds, and the estimate 
\eqref{eq:dcp-estimate} is true.
\end{theorem}

\begin{proof}
For any $\vv\in \vHcrlz$, let $\psi\in \Hoz$ be the solution of
\[
\binprod{\b \nab \psi}{\nab \phi} = \binprod{\b \vv}{\nab \phi},\quad \forall 
\phi \in \Hoz.
\]
It is easy to check that
\begin{equation}
\label{eq:dcp-gradient}
\norm{\b^{1/2}\nab \psi}_{\vLt} \leq \norm{\b^{1/2}\vv}_{\vLt} \leq 
\enorm{\vv}
\end{equation}
and that $\vw = \vv - \nab \psi$ satisfies
\begin{equation}
\label{eq:dcp-divfree}
\divv(\b \vw)= 0 \text{ in } \Om \quad\text{ and }\quad 
\vw\cross\vn = \vect{0} \text{ on } \p \Om.
\end{equation}
The decomposition $\vv = \vw+ \nab \psi$ shares the same form of the result 
\eqref{eq:dcp} we want to prove, yet the rest is to show that $\vw \in 
\vPH{1}$. 
To this end, we first construct an $\vH^1$-lifting of the $\vw$. Using 
integration by parts we have
\[
-\int_{\p\Om} \nab \phi\cdot (\vw\cross\vn)\,dS 
= \int_{\Om} \nab\phi\cdot \curlt \vw
= \int_{\p \Om} \phi\curlt \vw\cdot \vn\,dS \quad \forall \phi\in \Ho,
\]
thus $\vw\cross\vn = \vect{0}$ on $\p \Om$ implies $\curlt \vw\cdot \vn = 0$ on 
$\p \Om$ by a density argument (e.g. see \cite{Amrouche-Bernardi-Dauge-Girault}). 
Applying Theorem 3.17 in \cite{Amrouche-Bernardi-Dauge-Girault} on $\curlt \vw$,  
there exists a $\vw_0\in \vX(\Om,1)$ such that
\[
\left\{
\begin{aligned}
\curlt \vw_0 &= \curlt \vw  &\text{in }\Om,
\\
\divv\vw_0 &= 0  &\text{in }\Om,
\\
\vw_0\cross \vn & = \vect{0}  &\text{on }\p\Om.
\end{aligned}
\right.
\]
Taking the convexity of the $\Om$ into account, an embedding result from 
Theorem 2.17 in \cite{Amrouche-Bernardi-Dauge-Girault} reads that 
$\vX(\Om,1)\hookrightarrow \vH^1(\Om)$. Thus $\vw_0\in  \vH^1(\Om)$.
Obviously,
\[
\curlt(\vw-\vw_0)= \vect{0}\text{ in } \Om  \quad\text{ and }\quad
(\vw-\vw_0)\cross\vn = \vect{0}\text{ on }\p \Om.
\]
The simply-connectedness of $\Om$ implies that there exists a $\z \in \Ho$ (see 
Lemma \ref{lem:curlfreepotential}) with a constant boundary value such that
\[
\vw - \vw_0 = \nab \z \text{ in } \Om.
\]
By the fact that $\vw\at{\Om_j}$ is divergence free within each 
$\Om_j$ respectively, and $\jump{\b\vw\cdot \vn}{F} = 0$ for any $F\subset \fI$, 
one can check that the variational problem that $\z$ satisfies is
\[
\binprod{\b \nab \z}{\nab \phi} = -\sum_{F\subset \fI} \int_F\jump{\b\vw_0\cdot 
\vn}{F}\phi\,dS \quad \forall \phi\in \Hoz.
\]
Noticing 
$\jump{\b\vw_0\cdot\vn}{F} = \jump{\b}{F} (\vw_0\cdot\vn)\at{F}\in H^{1/2}(F)$ on 
any $F\subset \fI$, the regularity result of Theorem 4.1 in 
\cite{Costabel-Dauge-Nicaise} shows that $\z\in \PHs{2}$, in which a function is 
piecewisely $H^2$ smooth, while has $H^1$ regularity across the 
interfaces on the whole domain. This, in turn, implies that 
\[
\vw = \vw_0+ \nab \z \in \vPH{1}\cap\{\vv\in \vX(\Om,\b): \divv(\b \vv ) = 0 \}.
\]
Lastly, to prove the estimate, by the triangle inequality 
and~\eqref{eq:dcp-gradient}, we have
\[
\norm{\b^{1/2} \vw}_{\vLt}\leq \norm{\b^{1/2} \vv}_{\vLt} 
+ \norm{\b^{1/2} \nab \psi}_{\vLt} \leq C\enorm{\vv}.
\] 
It follows from Assumption~\ref{asp:coeff} (ii) and 
Lemma~\ref{lem:normeqiv}
that
\[
\sum^m_{j=1}\norm{\mu^{-1/2}\nab \vw}_{\vL^2(\Om_j)} \leq C 
\sum^m_{j=1}\norm{\b^{1/2}\nab \vw}_{\vL^2(\Om_j)} = C \norm{\b^{1/2}\curlt 
\vw}_{\vL^2(\Om)} \leq  C \enorm{\vv}.
\]
These inequalities and~\eqref{eq:dcp-gradient} imply the validity 
of~\eqref{eq:dcp-estimate} and, hence, it completes the proof of the 
theorem. 
\end{proof}

\begin{remark}
The decomposition result in Theorem {\em \ref{thm:dcp-general}} resembles 
that of Theorem 3.5 in \cite{Costabel-Dauge-Nicaise}: any vector field in 
$\vX(\Om,\b)$ can be split into a $\vPH{1}$-regular part, and a singular part 
solving a Dirichlet boundary problem $-\divv(\b\nab \psi) = f\in \Lt$. In the 
proof of Theorem {\em \ref{thm:dcp-general}}, we refine the results to cater 
the need for the pipeline of proving the reliability of the error estimator. 
Namely, when certain assumption of geometry is imposed, if a vector field 
$\vv\in \vX(\Om,\b)$ with its tangential trace vanishing on the boundary, and 
$\divv(\b \vv) = 0$, that singular part is non-existent.
\end{remark}


\end{document}